\documentclass[a4paper]{amsart}

\usepackage[T1]{fontenc}
\usepackage{lmodern}
\usepackage[utf8]{inputenc}
\usepackage{amssymb}
\usepackage[all]{xy}
\usepackage{enumitem}
\usepackage{mathrsfs,dsfont,mathtools}
\usepackage{nicefrac}

\usepackage[pdftitle={Homological algebra in bivariant K-theory and other triangulated categories. II},
  pdfauthor={Ralf Meyer},
  pdfsubject={Mathematics; MSC 18E30, 19K35, 46L80, 55U35}
]{hyperref}
\newcommand*{\MRref}[2]{\href{http://www.ams.org/mathscinet-getitem?mr=#1}{MR \textbf{#1}}}
\newcommand*{\arxiv}[1]{\href{http://www.arxiv.org/abs/#1}{arXiv: #1}}

\usepackage[lite]{amsrefs}
\usepackage{microtype}

\numberwithin{equation}{section}
\theoremstyle{plain}
\newtheorem{theorem}[equation]{Theorem}
\newtheorem{lemma}[equation]{Lemma}
\newtheorem{proposition}[equation]{Proposition}

\theoremstyle{definition}
\newtheorem{definition}[equation]{Definition}
\newtheorem{notation}[equation]{Notation}
\theoremstyle{remark}
\newtheorem{remark}[equation]{Remark}
\newtheorem{example}[equation]{Example}

\newcommand*{\C}{\mathbb C}
\newcommand*{\Z}{\mathbb Z}
\newcommand*{\N}{\mathbb N}

\newcommand*{\Left}{\mathbb L}
\newcommand*{\Right}{\mathbb R}

\newcommand*{\KK}{\textup{KK}}
\newcommand*{\K}{\textup{K}}

\newcommand*{\Bad}{\textup{\sffamily Bad}}

\newcommand*{\ID}{\textup{id}}

\newcommand*{\op}{\textup{op}}
\newcommand*{\Cst}{\textup C^*}

\newcommand*{\contra}{G} 

\newcommand*{\Ab}{\mathfrak{Ab}} 
\newcommand*{\KKcat}{\mathfrak{KK}}
\newcommand*{\RKKcat}{\mathfrak{RKK}}

\newcommand*{\CONT}{\textup C} 
\newcommand*{\Cat}{\mathfrak C}  
\newcommand*{\Addi}{\mathfrak C} 
\newcommand*{\Abel}{\mathfrak A} 
\newcommand*{\Tri}{\mathfrak T}  
\newcommand*{\Ideal}{\mathfrak I}
\newcommand*{\Null}{\mathfrak N} 
\newcommand*{\Local}{\mathfrak L}
\newcommand*{\Proj}{\mathfrak P} 

\newcommand*{\CC}{\mathcal{CC}}
\newcommand*{\CI}{\mathcal{CI}}
\newcommand*{\VC}{\mathcal{I}}
\newcommand*{\Fam}{\mathcal{F}} 

\newcommand*{\EG}{\mathcal E}

\newcommand*{\nb}{\nobreakdash}  
\newcommand*{\alb}{\hspace{0pt}} 

\mathtoolsset{mathic}
\DeclarePairedDelimiter{\abs}{\lvert}{\rvert}

\DeclarePairedDelimiter{\gen}{\langle}{\rangle}

\DeclareMathOperator{\hoinjlim}{ho-\varinjlim}

\DeclareMathOperator{\range}{range}
\DeclareMathOperator{\Ext}{Ext}
\DeclareMathOperator{\Hom}{Hom}
\DeclareMathOperator{\Res}{Res}
\DeclareMathOperator{\Ind}{Ind}
\DeclareMathOperator{\Ho}{Ho}
\DeclareMathOperator{\Rep}{Rep}

\newcommand*{\inOb}{\mathrel{\in\in}}
\newcommand*{\defeq}{\mathrel{\vcentcolon=}}

\newcommand*{\into}{\rightarrowtail}
\newcommand*{\prto}{\twoheadrightarrow}
\newcommand*{\lad}{\vdash}
\newcommand*{\rcross}{\mathbin{\ltimes_\textup{r}}}

\newcommand*{\blank}{\text{\textvisiblespace}}

\begin{document}

\title[Homological algebra in bivariant K-theory]{Homological algebra in bivariant K-theory and other triangulated categories. II}

\author{Ralf Meyer}
\email{rameyer@uni-math.gwdg.de}

\address{Mathematisches Institut\\
  Georg-August Universit\"at G\"ottingen\\
  Bunsenstra\ss{}e 3--5\\
  37073 G\"ottingen\\
  Germany}

\begin{abstract}
  We use homological ideals in triangulated categories to get a
  sufficient criterion for a pair of subcategories in a
  triangulated category to be complementary.  We apply this
  criterion to construct the Baum--Connes assembly map for
  locally compact groups and torsion-free discrete quantum
  groups.  Our methods are related to the abstract version of
  the Adams spectral sequence by Brinkmann and Christensen.
\end{abstract}

\subjclass[2000]{18E30, 19K35, 46L80, 55U35}

\maketitle

\section{Introduction}
\label{sec:intro}

The framework of \emph{triangulated categories} is ideal to
extend basic constructions from homotopy theory to categories
of \(\Cst\)\nb-algebras.  It provides a uniform setting for
various problems in non-commutative topology, including
homotopy colimits and Mayer--Vietoris sequences, universal
coefficient theorems, and generalisations of the Baum--Connes
assembly map (see \cites{Meyer-Nest:BC,
  Meyer-Nest:Homology_in_KK, Meyer-Nest:BC_Coactions,
  Meyer-Nest:Bootstrap, Meyer-Nest:Filtrated_K}).  More
specifically, the Baum--Connes assembly map for coactions of
certain compact Lie groups, which is studied
in~\cite{Meyer-Nest:BC_Coactions}, is always an isomorphism and
it is closely related to a universal coefficient theorem for
equivariant Kasparov theory by Jonathan Rosenberg and Claude
Schochet (\cite{Rosenberg-Schochet:Kunneth}).  Universal
coefficient theorems for Kirchberg's bivariant \(\K\)\nb-theory
for \(\Cst\)\nb-algebras over certain finite topological spaces
are derived in \cites{Meyer-Nest:Bootstrap,
  Meyer-Nest:Filtrated_K}.

This article continues~\cite{Meyer-Nest:Homology_in_KK}, which
deals with a framework for carrying over familiar notions from
homological algebra to general triangulated categories.  Before
we explain what \emph{this} article is about, we outline some
important ideas from~\cite{Meyer-Nest:Homology_in_KK}.

The localisation (or total derived functor) of an additive
functor between Abelian categories is a functor between their
derived categories.  Mapping chain complexes to chain
complexes, it belongs to the world of triangulated categories
by definition.  Although the more classical derived functors
originally live in the underlying Abelian categories, they can
be carried over to triangulated categories as well.

Both localisations and derived functors require additional
structure on a triangulated category to be defined.  For the
localisation of a functor, we specify the \emph{subcategory} to
localise at, consisting of all \emph{objects} on which the
localisation vanishes.  For its \emph{derived functors}, we
specify an \emph{ideal}, consisting of all \emph{morphisms} on
which the derived functors vanish.

The idea to use ideals in triangulated categories goes back to
Daniel Christensen~\cite{Christensen:Ideals}.  Some important
related concepts are due to Apostolos
Beligiannis~\cite{Beligiannis:Relative}, who uses a slightly
different but equivalent setup, which is inspired by the notion
of an exact category in homological algebra.

A \emph{homological ideal} in a triangulated category~\(\Tri\)
is, by definition, the kernel of a stable homological functor
(see~\cite{Meyer-Nest:Homology_in_KK}).  Such an
ideal~\(\Ideal\) allows us to carry over various notions of
homological algebra to~\(\Tri\).  The ultimate explanation for
this is that a homological ideal generates a canonical
homological functor to a certain Abelian category, namely, the
\emph{universal} \(\Ideal\)\nb-exact stable homological functor
\(H_\Ideal\colon \Tri\to\Abel_\Ideal\Tri\).  All homological
notions in~\(\Tri\) defined using the ideal~\(\Ideal\) reflect
familiar notions in this Abelian category.  The homological
algebra in the target Abelian category \(\Abel_\Ideal\Tri\)
provides a rough \emph{Abelian approximation} to the
category~\(\Tri\).

An interesting and typical example is the \(G\)\nb-equivariant
Kasparov category \(\KKcat^G\) for a countable discrete
group~\(G\).  Let~\(\Ideal\) be the ideal defined by the
\(\K\)\nb-theory functor, that is, an element of \(\KK^G(A,B)\)
belongs to the ideal if it induces the zero map
\(\K_*(A)\to\K_*(B)\).  The resulting Abelian approximation
\(\Abel_\Ideal(\KKcat^G)\) is the category of all
\(\Z/2\)\nb-graded countable modules over the group ring
\(\Z[G]\), and the universal functor maps a
\(\Cst\)\nb-algebra~\(A\) with an action of~\(G\) to its
\(\K\)\nb-theory, equipped with the induced action of~\(G\)
(this is a special case of a result
in~\cite{Meyer-Nest:Homology_in_KK}).

Notice that the passage to the universal functor adds the group
action on \(\K_*(A)\).  Forgetting this group action does not
change the ideal defined by the functor, but it kills most
interesting homological algebra.  (In
Section~\ref{sec:applications_BC}, we will actually consider a
smaller ideal in \(\KKcat^G\) that is more closely related to
the Baum--Connes conjecture, but leads to a more complicated
Abelian approximation.)

The Abelian category \(\Abel_\Ideal\Tri\) is usually not a
localisation of~\(\Tri\): we must modify both morphisms and
objects to get an Abelian category.  Instead, it is described
in~\cite{Beligiannis:Relative} as a localisation of the Abelian
category containing~\(\Tri\) constructed by Peter Freyd.  The
main innovation in~\cite{Meyer-Nest:Homology_in_KK} is a
concrete criterion for a stable homological functor to be
universal, which involves its partially defined left adjoint.
Using this criterion, we can often find rather concrete models
for the universal functor -- as in the example mentioned
above -- and then compute derived functors associated to the
ideal.

What do the derived functors of a homological functor
on~\(\Tri\) tell us about the original functor?  In general,
these derived functors are always related to the original
functor by a spectral sequence, whose convergence we will
discuss below.  This result is mainly of theoretical importance
because spectral sequence computations are almost impossible
without additional simplifying assumptions.  But given how much
information is lost by passing to an Abelian category, we
cannot hope for much more than a spectral sequence.

The spectral sequence that links a functor to its derived
functors was already discovered in the 1960s before
triangulated categories became popular.  First Frank Adams
treated an important special case in stable homotopy
theory -- the Adams spectral
sequence~\cite{Adams:Stable_homotopy}.  This was reformulated
in an abstract setting by Hans-Berndt
Brinkmann~\cite{Brinkmann:Relative}.  Daniel
Christensen~\cite{Christensen:Ideals} formulated the Adams
spectral sequence in the setting of triangulated categories,
apparently unaware of Brinkmann's work.

Given the sources of the spectral sequence, we call it the
\emph{ABC spectral sequence} here.  We describe its
construction and its higher pages in greater detail than
previous authors and weaken the assumptions needed to guarantee
its convergence.

I was drawn towards this theory because similar ideas provide
an effective method to prove that pairs of subcategories are
complementary; this is the most difficult technical aspect of
the construction of the Baum--Connes assembly map
in~\cite{Meyer-Nest:BC}.  In Section~\ref{sec:applications_BC},
we first apply our new criterion to the group case already
treated in~\cite{Meyer-Nest:BC} and then define an analogue of
the Baum--Connes assembly map for all ``torsion-free'' discrete
quantum groups.  More precisely, we construct an assembly map
for all discrete quantum groups, but since this map does not
take into account torsion, it is not the right analogue of the
Baum--Connes assembly map unless the quantum group in question
is torsion-free.

A built-in feature of our new assembly map is that its domain
is computed by a spectral sequence -- the ABC spectral sequence
-- whose second page is quite accessible.  The spectral
sequence computation is very difficult, but an operator
algebraist might consider it to be a topological problem, that
is, Someone Else's Problem.  His own problem is to find out
when the assembly map is an isomorphism.  Given our experience
with the group case, this should happen often but not always.
So far -- besides classical groups -- only the duals of certain
compact Lie groups and quantum \(\textup{SU}(2)\) have been
treated in~\cite{Meyer-Nest:BC_Coactions} and
in~\cite{Voigt:Baum-Connes_qSU2}, respectively.  For the
alternative approach by Aderemi Kuku and Debashish Goswami
in~\cite{Goswami-Kuku:BC_quantum}, it is unclear whether the
domain of the assembly map is computable by topological
methods.

Our criterion for complementarity of two subcategories is also
useful in situations that have nothing to do with bivariant
\(\K\)\nb-theory.  The improvement upon similar criteria
in~\cite{Beligiannis:Relative} is that we can cover categories
that are not compactly generated: what we need is an ideal with
enough projective objects that is compatible with countable
direct sums.  This assumption is still satisfied for the ideals
that appear in connection with the Baum--Connes assembly map,
although the categories in question are probably not compactly
generated.

More precisely, the criterion is the following.  Let~\(\Ideal\)
be a homological ideal in a triangulated category~\(\Tri\).  We
assume that~\(\Tri\) has countable direct sums and that the
ideal~\(\Ideal\) is compatible with countable direct sums in a
suitable sense.  Furthermore, we assume that~\(\Ideal\) has
enough projective objects.  Let \(\Proj_\Ideal\subseteq\Tri\)
be the class of \(\Ideal\)\nb-projective objects in~\(\Tri\)
and let \(\gen{\Proj_\Ideal}\) be the localising subcategory
generated by it, that is, the smallest triangulated subcategory
that is closed under countable direct sums and
contains~\(\Proj_\Ideal\).  Finally, let~\(\Null_\Ideal\) be
the subcategory of \(\Ideal\)\nb-contractible objects.  Under
the assumptions above, the pair of subcategories
\((\gen{\Proj_\Ideal},\Null_\Ideal)\) is complementary, that
is, \(\Tri_*(P,N)=0\) whenever \(P\inOb\gen{\Proj_\Ideal}\) and
\(N\inOb\Null_\Ideal\), and any object \(A\inOb\Tri\) is part
of an exact triangle \(P\to A\to N\to P[1]\) with
\(P\inOb\gen{\Proj_\Ideal}\) and \(N\inOb\Null_\Ideal\).
Equivalently, the subcategory~\(\Null_\Ideal\) is
\emph{reflective}, that is, the embedding
\(\Null_\Ideal\to\Tri\) has a right adjoint functor.

Our proof also provides the following structural information on
the category \(\gen{\Proj_\Ideal}\).  First, we get an
increasing chain \((\Proj_\Ideal^n)_{n\in\N}\) of
subcategories, consisting of the projective objects for the
ideals~\(\Ideal^n\); these can also be generated iteratively
from~\(\Proj_\Ideal\) using exact triangles.  We show that any
object of \(\gen{\Proj_\Ideal}\) is a homotopy colimit of an
inductive system~\(P_n\) with \(P_n\inOb\Proj_\Ideal^n\).

\begin{notation}
  \label{note:inOb}
  We write \(f\in\Cat\) for a morphism and \(A\inOb\Cat\) for
  an object of a category~\(\Cat\).  We denote the category of
  Abelian groups by~\(\Ab\).  We usually write~\(\Tri\) for
  triangulated, \(\Abel\) for Abelian, and~\(\Addi\) for
  additive categories.  The translation automorphism in a
  triangulated category is denoted by \(A\mapsto A[1]\).
\end{notation}

\section{Homological ideals, powers, and filtrations}
\label{sec:powers_filtrations}

The convergence of a spectral sequence always involves a
filtration on the limit group.  Hence we expect a homological
ideal~\(\Ideal\) in~\(\Tri\) to generate filtrations on the
category~\(\Tri\) itself and on homological and cohomological
functors on~\(\Tri\).  After recalling some basic notions, we
introduce these filtrations here.

We will use the results and the notation
of~\cite{Meyer-Nest:Homology_in_KK}.  In particular, a
\emph{stable} category is a category with a \emph{translation}
automorphism, denoted \(A\mapsto A[1]\), and a \emph{stable}
functor is a functor~\(F\) together with natural isomorphisms
\(F(A[1]) \cong (FA)[1]\) for all objects~\(A\).

Let \(F\colon \Tri\to\Abel\) be a stable homological functor
from a triangulated category~\(\Tri\) to a stable Abelian
category~\(\Abel\).  We define an ideal \(\ker F\) in~\(\Tri\)
by
\[
\ker F(A,B) \defeq \{\varphi\in\Tri(A,B) \mid F(\varphi)=0\}.
\]
Ideals of this form are called \emph{homological ideals}.  A
homological ideal is used in~\cite{Meyer-Nest:Homology_in_KK}
to carry over various notions from Abelian to triangulated
categories.  This includes \(\Ideal\)\nb-epimorphisms,
\(\Ideal\)\nb-exact chain complexes, \(\Ideal\)\nb-exact
functors, \(\Ideal\)\nb-projective objects, and
\(\Ideal\)\nb-projective resolutions.  The first three of these
can be tested using the functor~\(F\); for instance, a chain
complex with entries in~\(\Tri\) is \(\Ideal\)\nb-exact if and
only if its \(F\)\nb-image is an exact chain complex in the
Abelian category~\(\Abel\).  Projective objects can only be
described in terms of~\(F\) if~\(F\) is the \emph{universal}
\(\Ideal\)\nb-exact stable homological functor,
see~\cite{Meyer-Nest:Homology_in_KK}.  We also call a morphism
an \emph{\(\Ideal\)\nb-phantom map} if it belongs
to~\(\Ideal\).

Most of our constructions require~\(\Tri\) to contain
\emph{enough \(\Ideal\)\nb-projective objects} -- that is, any
object should be the range of an \(\Ideal\)\nb-epimorphism with
\(\Ideal\)\nb-projective domain.  This is equivalent to the
existence of \(\Ideal\)\nb-projective resolutions for all
objects.

\begin{remark}
  \label{rem:Christensen_projective_class}
  Daniel Christensen uses a somewhat different terminology
  in~\cite{Christensen:Ideals}.  His \emph{projective classes}
  \((\Ideal,\Proj)\) turn out to be the same as a homological
  ideal~\(\Ideal\) with enough projective objects together with
  its class \(\Proj=\Proj_\Ideal\) of projective objects.  The
  ideal~\(\Ideal\) in a projective class is homological
  because, in the presence of enough projective objects, the
  universal homological functor with kernel~\(\Ideal\) is
  well-defined.  There are two ways to construct this universal
  functor, which involve a localisation of categories in one
  step.  Apostolos Beligiannis~\cite{Beligiannis:Relative}
  first embeds the category~\(\Tri\) into an Abelian category
  and then localises the latter at a Serre subcategory.  The
  authors use the heart of a t\nb-structure on a suitable
  derived category of chain complexes over~\(\Tri\) in
  \cite{Meyer-Nest:Homology_in_KK}*{\S3.2.1}.  In both cases,
  the morphisms in the relevant localisations can be computed
  using projective resolutions, so that the localisation is
  again a category with morphism \emph{sets} instead of
  morphism classes.
\end{remark}

\subsection{Powers and intersections of ideals}
\label{sec:powers_intersection_ideal}

At first, we do not care whether the ideals we are dealing with
are homological.  Let~\(\Addi\) be an additive category.  If
\((\Ideal_\alpha)_{\alpha\in S}\) is a set of ideals, then the
intersection \(\bigcap \Ideal_\alpha\) is again an ideal.  If
\(\Ideal_1,\Ideal_2\subseteq\Addi\) are ideals, define
\[
\Ideal_1\circ\Ideal_2(A,B)
\defeq \{f_1\circ f_2\mid
\text{\(f_1\in\Ideal_1(X,B)\), \(f_2\in\Ideal_2(A,X)\)
  for some \(X\inOb \Addi\)}\}.
\]
This is a subgroup of \(\Addi(A,B)\) because we can decompose
\(f_1\circ f_2+f_1'\circ f_2'\) as
\[
\xymatrix@C+3mm{
  A \ar[r]^-
  {\Bigl(\begin{smallmatrix}f_2\\f_2'\end{smallmatrix}\Bigr)}&
  X\oplus X'
  \ar[r]^-{(\begin{smallmatrix}f_1&f_1'\end{smallmatrix})}&
  B.
}
\]
Thus \(\Ideal_1\circ\Ideal_2\) is an ideal in~\(\Addi\).  We
have \(\Ideal_1\circ\Ideal_2\subseteq \Ideal_1\cap\Ideal_2\).

Now the \emph{powers} of an ideal \(\Ideal\subseteq\Addi\) are
defined recursively: we let \(\Ideal^0\defeq\Addi\) consist of
all morphisms and define \(\Ideal^n \defeq
\Ideal^{n-1}\circ\Ideal\) for \(1\le n<\infty\).  The sequence
of ideals \((\Ideal^n)_{n\in\N}\) is decreasing, and we have
\(\Ideal^m\circ\Ideal^n=\Ideal^{m+n}\) for all \(m,n\in\N\).
If \(\Ideal^n=\Ideal^{n+1}\) for some \(n\in\N\), then
\(\Ideal^n=\Ideal^N\) for all \(N\ge n\).

We also let \(\Ideal^\infty\defeq \bigcap_{n\in\N} \Ideal^n\)
and \(\Ideal^{n\infty} \defeq (\Ideal^\infty)^n\).  In general,
the ideals \(\Ideal\circ\Ideal^\infty\) and
\(\Ideal^\infty\circ\Ideal\) may differ from~\(\Ideal^\infty\)
(see the remark after Proposition~\ref{pro:ABC_spese_dual}).
Theorem~\ref{the:cellular_limit_tower} shows that
\(\Ideal^{n\infty} = \Ideal^{2\infty}\) for all \(n\ge2\)
if~\(\Ideal\) is compatible with countable direct sums.

Now we replace the additive category~\(\Addi\) by a
triangulated category~\(\Tri\) and restrict attention to
homological ideals.  It is not obvious whether the powers of a
homological ideal are again homological.  If \(\Ideal = \ker
F\), then a functor with kernel \(\Ideal^2\subset\Ideal\)
contains more information than~\(F\) because it has a smaller
kernel.  Therefore, we cannot hope to construct such a functor
out of~\(F\).

Nevertheless, I expect that products and intersections of
homological ideals are again homological, at least if the
categories in question are small to rule out set theoretic
difficulties with localisation of categories.  A proof could
use Beligiannis' axiomatic characterisation of homological
ideals.  Since we only need the much easier case where there
are enough projective objects, I have not completed the
argument.  Proposition~2.5 and Theorem~3.1
in~\cite{Beligiannis:Relative} show that our ``homological
ideals'' are exactly the ``saturated \(\Sigma\)\nb-stable
ideals'' in Beligiannis' notation.  Clearly, products and
intersections of \(\Sigma\)\nb-stable ideals remain
\(\Sigma\)\nb-stable, and intersections of saturated ideals
remain saturated.  It is less clear whether products of
saturated ideals remain saturated; the proof should involve the
octahedral axiom.

Here we only consider the easy case of ideals with enough
projective objects, where we can describe which objects are
projective for products and intersections:

\begin{proposition}[\cite{Christensen:Ideals}*{Proposition 3.3}]
  \label{pro:product_class}
  Let \(\Ideal_1\) and~\(\Ideal_2\) be homological ideals
  in~\(\Tri\) with enough projective objects.  Then
  \(\Ideal_1\circ\Ideal_2\) is a homological ideal with enough
  projective objects.  An object~\(A\) of~\(\Tri\) is
  \(\Ideal_1\circ\Ideal_2\)-projective if and only if there are
  \(\Ideal_j\)\nb-projective objects~\(P_j\) and an exact
  triangle \(P_2\to P\to P_1\to P_2[1]\), such that~\(A\) is a
  direct summand of~\(P\).
\end{proposition}

\begin{proposition}[see \cite{Christensen:Ideals}*{Proposition
    3.1}]
  \label{pro:intersect_class}
  Let \((\Ideal_\alpha)_{\alpha\in S}\) be a set of homological
  ideals in~\(\Tri\) with enough projective objects.  Suppose
  that~\(\Tri\) has direct sums of cardinality~\(\abs{S}\).
  Then \(\Ideal_S\defeq \bigcap_{\alpha\in S} \Ideal_\alpha\)
  is a homological ideal with enough projective objects.  An
  object~\(A\) of~\(\Tri\) is \(\Ideal_S\)\nb-projective if and
  only if there are \(\Ideal_\alpha\)\nb-projective
  objects~\(P_\alpha\) such that~\(A\) is a direct summand of
  \(\bigoplus_{\alpha\in S} P_\alpha\).
\end{proposition}

We may use Christensen's results because of
Remark~\ref{rem:Christensen_projective_class}.

\begin{definition}
  \label{def:Proj_Ideal}
  Let~\(\Ideal\) be a homological ideal in a triangulated
  category~\(\Tri\) with enough projective objects.  We
  write~\(\Proj_\Ideal\) for the class of
  \(\Ideal\)\nb-projective objects, and~\(\Proj_\Ideal^n\) for
  the class of \(\Ideal^n\)\nb-projective objects for
  \(n\in\N\cup\{\infty\}\).
\end{definition}

The class~\(\Proj_\Ideal\) is always closed under direct
summands, suspensions, and direct sums that exist in~\(\Tri\).

Propositions \ref{pro:product_class}
and~\ref{pro:intersect_class} show that the powers~\(\Ideal^n\)
for \(n\in\N\cup\{\infty\}\) have enough projective objects.
Moreover, \(\Proj_\Ideal^n\) for \(n\in\N\) consists of all
direct summands of objects \(A_n\inOb\Tri\) for which there is
an exact triangle \(A_{n-1}\to A_n\to A_1\to A_{n-1}[1]\) with
\(A_{n-1}\inOb\Proj_\Ideal^{n-1}\), \(A_1\inOb\Proj_\Ideal\);
and \(\Proj_\Ideal^\infty\) consists of all retracts of objects
of the form \(\bigoplus_{n\in\N} A_n\) with
\(A_n\inOb\Proj_\Ideal^n\).  The phantom castle introduced in
Definition~\ref{def:phantom_castle} explicitly decomposes
objects of~\(\Proj_\Ideal^n\) into objects of~\(\Proj_\Ideal\);
essentially, its construction is the proof of
Proposition~\ref{pro:product_class}.

\begin{example}
  \label{exa:idempotent_ideals}
  Let~\(\Ideal\) be a homological ideal with enough projective
  objects.  If \(\Ideal=\Ideal^2\), then
  Proposition~\ref{pro:product_class} implies
  that~\(\Proj_\Ideal\) is closed under extensions.  Since this
  subcategory is always closed under direct summands,
  suspensions, isomorphism, and direct sums, \(\Proj_\Ideal\)
  is a localising subcategory of~\(\Tri\).

  Conversely, if~\(\Proj_\Ideal\) is a triangulated
  subcategory, then \(\Ideal\) and~\(\Ideal^2\) have the same
  projective objects.  Since an ideal with enough projective
  objects is determined by its class of projective objects,
  this implies \(\Ideal=\Ideal^2\).

  Homological ideals with \(\Ideal=\Ideal^2\), but possibly
  without enough projective objects, play an important role
  in~\cite{Krause:Cohomological_quotients} as a substitute for
  localising subcategories.
\end{example}

We usually know very little about the Abelian approximations
generated by~\(\Ideal^n\) for \(n\ge2\), even if the situation
for~\(\Ideal\) itself is rather simple.  Derived functors for
\(\Ideal\) and~\(\Ideal^2\) do not seem closely related.  This
is particularly obvious in cases where \(\Ideal\neq0\) and
\(\Ideal^2=0\).  For instance, this happens if~\(\Ideal\) is
the kernel of the homology functor on the derived category of
the category of Abelian groups.  Here the universal
\(\Ideal\)\nb-exact functor is the homology functor to
\(\Ab^\Z\); the universal \(\Ideal^2\)\nb-exact functor is the
Freyd embedding of the derived category into an Abelian
category.

\subsection{The phantom filtrations}
\label{sec:phantom_filtrations}

Let~\(\Ideal\) be an ideal in an additive category~\(\Addi\).
Since \(\Ideal^\alpha\subseteq\Ideal^\beta\) for
\(\alpha\ge\beta\), we get a decreasing filtration
\[
\Addi(A,B)
= \Ideal^0(A,B)
\supseteq \Ideal^1(A,B)
\supseteq \Ideal^2(A,B)
\supseteq \dotsb
\supseteq \Ideal^\infty(A,B)
\supseteq \{0\},
\]
called the \emph{phantom
  filtration}~\cite{Beligiannis:Relative}.  We shall also need
related filtrations on contravariant and covariant functors
on~\(\Addi\).

Let \(\contra\colon \Addi^\op\to\Ab\) be a \emph{contravariant}
functor and \(A\inOb\Addi\).  We define a decreasing filtration
\[
\contra(A) = \Ideal^0 \contra(A)
\supseteq \Ideal^1 \contra(A)
\supseteq \Ideal^2 \contra(A)
\supseteq \dotsb
\supseteq \Ideal^\infty \contra(A)
\supseteq \{0\}
\]
on \(\contra(A)\) by
\[
\Ideal^\alpha \contra(A)
\defeq \{f^*(\xi) \mid
\text{\(f\in\Ideal^\alpha(A,B)\), \(\xi\in \contra(B)\)
for some \(B\inOb \Addi\)}\}.
\]
If we apply this construction to the representable functor
\(\Addi(\blank,B)\) we get back the filtration
\(\Ideal^\alpha(A,B)\) on \(\Addi(A,B)\).  If~\(\contra\) is
compatible with direct sums,
then~\eqref{eq:contra_infty_intersect} asserts that
\(\bigcap_{n\in\N} \Ideal^n\contra(A) =
\Ideal^\infty\contra(A)\).

The functoriality of~\(\contra\) restricts to maps
\[
\Ideal^\beta(A,B)\otimes \Ideal^\alpha \contra(B) \to
\Ideal^{\alpha+\beta} \contra(A),\qquad
f\otimes x\mapsto f^*(x),
\]
for all \(\alpha,\beta\).  In particular, \(\Ideal^\alpha
\contra\) is a contravariant functor on~\(\Addi\).  The
ideal~\(\Ideal^\beta\) acts trivially on the subquotients
\(\Ideal^\alpha \contra(A)/\Ideal^{\alpha+\beta}
\contra(A)\), which therefore descend to functors on the
quotient category \(\Addi/\Ideal^\beta\).

We may also view~\(\contra\) as a right module over the
category~\(\Addi\) and \(\Addi/\Ideal^\alpha\) as a
\(\Addi\)\nb-bimodule.  The quotient
\(\contra/\Ideal^\alpha\contra\) corresponds to the right
\(\Addi\)\nb-module \(\contra\otimes_\Addi
\Addi/\Ideal^\alpha\).

For a covariant functor \(F\colon \Cat\to\Ab\), we define an
increasing filtration
\[
\{0\} = F:\Ideal^0(A)
\subseteq F:\Ideal^1(A)
\subseteq F:\Ideal^2(A)
\subseteq \dotsb
\subseteq F:\Ideal^\infty(A)
\subseteq F(A)
\]
for any \(A\inOb\Addi\) by
\[
F:\Ideal^\alpha(A)
\defeq \{x\in F(A)\mid \text{\(f_*(x)=0\)
for all \(f\in\Ideal^\alpha(A,B)\), \(B\inOb \Addi\)}\}.
\]
If \(\Ideal\) and~\(F\) are compatible with direct sums, then
\(F:\Ideal^\infty(A) = \bigcup_{n\in\N} F:\Ideal^n(A)\) (see
Theorem~\ref{the:homological_limit}), but this need not be the
case in general.

The functoriality of~\(F\) restricts to maps
\[
\Ideal^\beta(A,B)\otimes F:\Ideal^{\alpha+\beta} (B) \to
F:\Ideal^\alpha (A),\qquad
f\otimes x\mapsto f_*(x),
\]
for all \(\alpha,\beta\).  In particular, \(F:\Ideal^\alpha\)
is a covariant functor on~\(\Addi\).  The
ideal~\(\Ideal^\beta\) acts trivially on the subquotients
\(F:\Ideal^{\alpha+\beta}(A) \bigm/ F:\Ideal^\alpha(A)\), which
therefore descend to functors on \(\Addi/\Ideal^\beta\).

We may also view~\(F\) as a left module over the
category~\(\Addi\) and \(\Addi/\Ideal^\alpha\) as a
\(\Addi\)\nb-bimodule.  Then \(F:\Ideal^\alpha\) corresponds to
the left \(\Addi\)\nb-module
\(\Hom_\Addi(\Addi/\Ideal^\alpha,F)\).

The filtration \(F:\Ideal^\alpha(A)\) is closely related to
projective resolutions of~\(A\).  In contrast, the filtration
\[
\Ideal^\alpha F(A) \defeq \{f_*(\xi)\in F(A)\mid
\text{\(f\in\Ideal^\alpha(B,A)\), \(\xi\in F(B)\)
for some \(B\inOb \Addi\)}\}
\]
is related to \emph{injective} (co)resolutions.  There is also
an increasing filtration \(\contra:\Ideal^n\) for a
contravariant functor~\(\contra\).  The filtrations
\(\Ideal^\alpha F\) and \(\contra:\Ideal^\alpha\) will not be
used in this article.

\section{From projective resolutions to complementary pairs}
\label{sec:pro_res_complementary}

First we refine a projective resolution by adjoining certain
phantom maps.  This yields the \emph{phantom tower} over an
object (see also~\cite{Beligiannis:Relative}).  We show that
the projective resolution determines this tower uniquely up to
non-canonical isomorphism.  There is another tower over an
object, the \emph{cellular approximation tower}.  These two
towers are related by various commuting diagrams and exact
triangles; we call the collection of all these exact or
commuting triangles the \emph{phantom castle}.

The goal of this section is to show that the categories
\(\gen{\Proj_\Ideal}\) and~\(\Null_\Ideal\) are complementary
if~\(\Ideal\) is compatible with direct sums.  Before we come
to that, we recall the notion of complementary pair of
subcategories and define what it means for an ideal to be
compatible with countable direct sums.  The main ingredients in
the proof are the homotopy colimits of the phantom tower and
the cellular approximation tower.  Finally, we describe a
method for checking that a given localising subcategory is
reflective, that is, part of a complementary pair.

All results involving infinite direct sums require that the
category~\(\Tri\) has countable direct sums.  Triangulated
categories involving bivariant \(\K\)\nb-theory have no more
than countable direct sums because of built-in separability
assumptions that make the analysis behind the scenes work.
This is why we only use \emph{countable} direct sums.  Of
course, everything remains true if we drop the word
``countable'' or replace it by another cardinality constraint.

A triangulated subcategory of~\(\Tri\) is called
\emph{localising} (more precisely,
\emph{\(\aleph_0\)\nb-localising}) if it is closed under
countable direct sums.  Localising subcategories are
automatically thick, that is, closed under direct summands
(see~\cite{Neeman:Triangulated}).

Let~\(\Ideal\) be a homological ideal in a triangulated
category~\(\Tri\).  Recall that an
\emph{\(\Ideal\)\nb-projective resolution} of an object~\(A\)
of~\(\Tri\) is a chain complex
\[
\dotsb \xrightarrow{\delta_{n+1}}
P_n \xrightarrow{\delta_n}
P_{n-1} \xrightarrow{\delta_{n-1}}
\dotsb \xrightarrow{\delta_2}
P_1 \xrightarrow{\delta_1}
P_0
\]
of \(\Ideal\)\nb-projective objects~\(P_n\), augmented by a map
\(\pi_0\colon P_0\to A\), such that the augmented chain complex
is \(\Ideal\)\nb-exact.  If \(\Ideal=\ker F\) for a stable
homological functor~\(F\) to some Abelian category~\(\Abel\),
then \emph{\(\Ideal\)\nb-exactness} means that the chain
complex
\[
\dotsb \xrightarrow{F(\delta_{n+1})}
F(P_n) \xrightarrow{F(\delta_n)}
F(P_{n-1}) \xrightarrow{F(\delta_{n-1})}
\dotsb \xrightarrow{F(\delta_1)}
F(P_0) \xrightarrow{F(\pi_0)} F(A)
\]
is exact in~\(\Abel\).  We say that~\(\Ideal\) has \emph{enough
  projective objects} if each \(A\inOb\Tri\) has such an
\(\Ideal\)\nb-projective resolution.

\subsection{The phantom tower}
\label{sec:phantom_tower}

\begin{definition}
  \label{def:phantom_tower}
  A \emph{phantom tower} over an object~\(A\) of~\(\Tri\) is a
  diagram
  \begin{equation}
    \label{eq:phantom_tower}
    \begin{gathered}
      \xymatrix@C=1em{
        A \ar@{=}[r]&
        N_0 \ar[rr]^{\iota_0^1}&&
        N_1 \ar[rr]^{\iota_1^2} \ar[dl]|-\circ^{\epsilon_0}&&
        N_2 \ar[rr]^{\iota_2^3} \ar[dl]|-\circ^{\epsilon_1}&&
        N_3 \ar[rr] \ar[dl]|-\circ^{\epsilon_2}&&
        \cdots \ar[dl]|-\circ\\
        &&P_0 \ar[ul]^{\pi_0}&&
        P_1 \ar[ul]^{\pi_1} \ar[ll]|-\circ^{\delta_1}&&
        P_2 \ar[ul]^{\pi_2} \ar[ll]|-\circ^{\delta_2}&&
        P_3 \ar[ul]^{\pi_3} \ar[ll]|-\circ^{\delta_3}&&
        \cdots \ar[ul] \ar[ll]|-\circ
      }
    \end{gathered}
  \end{equation}
  with \(\Ideal\)\nb-phantom maps~\(\iota_n^{n+1}\) and
  \(\Ideal\)\nb-projective objects~\(P_n\) for \(n\in\N\), such
  that the triangles
  \[
  P_n \xrightarrow{\pi_n}
  N_n \xrightarrow{\iota_n^{n+1}}
  N_{n+1} \xrightarrow{\epsilon_n}
  P_n[1]
  \]
  in~\eqref{eq:phantom_tower} are exact for all \(n\in\N\) and
  the other triangles in~\eqref{eq:phantom_tower} commute, that
  is, \(\delta_{n+1} = \epsilon_n\circ\pi_{n+1}\) for all
  \(n\in\N\).  Notice our convention that circled arrows are
  maps of degree~\(1\).
\end{definition}

Since the maps~\(\delta_n\) in the phantom tower have
degree~\(1\), we slightly modify our notion of projective
resolution, letting the boundary maps have degree~\(1\).

\begin{lemma}
  \label{lem:phantom_tower}
  The maps~\(\delta_n\) for \(n\in\N_{\ge1}\) and~\(\pi_0\) in
  a phantom tower over~\(A\) form an \(\Ideal\)\nb-projective
  resolution of~\(A\).  Conversely, any
  \(\Ideal\)\nb-projective resolution can be embedded in a
  phantom tower, which is unique up to non-canonical
  isomorphism.

  A morphism \(f\colon A\to A'\) lifts
  \textup{(}non-canonically\textup{)} to a morphism between two
  given phantom towers over \(A\) and~\(A'\).  A chain map
  between projective resolutions of \(A\) and~\(A'\) extends to
  the phantom towers that contain these resolutions.
\end{lemma}

\begin{proof}
  Let \(P_n\), \(\pi_0\), and~\(\delta_n\) be part of a phantom
  tower over~\(A\).  The objects~\(P_n\) are
  \(\Ideal\)\nb-projective by definition, and
  \(\delta_n\circ\delta_{n+1}=0\) for all \(n\in\N\) and
  \(\pi_0\circ\delta_1=0\) because these products involve two
  consecutive arrows in an exact triangle.  Hence the maps
  \(\delta_n\) and~\(\pi_0\) form a chain complex.  We claim
  that it is \(\Ideal\)\nb-exact.

  Let~\(F\) be a stable homological functor with \(\ker
  F=\Ideal\).  Recall that a chain complex is
  \(\Ideal\)\nb-exact if and only if its \(F\)\nb-image is
  exact in the usual sense by
  \cite{Meyer-Nest:Homology_in_KK}*{Lemma 3.9}.  The exact
  triangles in the phantom tower yield short exact sequences
  \[
  F_{*+1}(N_{n+1}) \into F_*(P_n) \prto F_*(N_n)
  \]
  for all \(n\in\N\) because \(\iota_n^{n+1}\in\Ideal\); here
  \(F_*(A)\defeq F(A[-n])\).  Splicing these extensions as in
  the definition of the Yoneda product, we get an exact chain
  complex.  Since this chain complex is
  \[
  \dotsb \to
  F_{*+2}(P_2) \to
  F_{*+1}(P_1) \to
  F_{*}(P_0) \to
  F_*(A) \to
  0,
  \]
  we have got an \(\Ideal\)\nb-exact chain complex and hence an
  \(\Ideal\)\nb-projective resolution.

  Now let \(\pi_0\colon P_0\to A\) and \(\delta_n\colon P_n\to
  P_{n-1}[1]\) for \(n\in\N_{\ge1}\) form an
  \(\Ideal\)\nb-projective resolution.  We recursively
  construct the triangles that comprise the phantom tower.  To
  begin with, we embed~\(\pi_0\) in an exact triangle
  \[
  P_0 \xrightarrow{\pi_0} A \xrightarrow{\iota_0^1} N_1
  \xrightarrow{\epsilon_0} P_0[1].
  \]
  Since~\(\pi_0\) is \(\Ideal\)\nb-epic, \(\iota_0^1\) is an
  \(\Ideal\)\nb-phantom map and~\(\epsilon_0\) is
  \(\Ideal\)\nb-monic.  Thus our exact triangle yields a short
  exact sequence
  \[
  \Tri_{*+1}(P,N_1) \into \Tri_*(P,P_0) \prto \Tri_*(P,A)
  \]
  for any \(\Ideal\)\nb-projective object~\(P\).  In
  particular, this applies to \(P=P_1\) and shows
  that~\(\delta_1\) factors uniquely as
  \(\delta_1=\epsilon_0\circ\pi_1\) with
  \(\pi_1\in\Tri_0(P_1,N_1)\).

  We claim that~\(\pi_1\) is \(\Ideal\)\nb-epic.  Let~\(F\) be
  a defining functor for~\(\Ideal\) as above.  Then \(F(P_1)\to
  F(P_0)\to F(A)\) is exact at \(F(P_0)\), and \(F(N_1)\into
  F(P_0)\prto F(A)\) is a short exact sequence.  Hence the
  range of \(F(\delta_1)\) is isomorphic to \(F(N_1)\).  This
  implies that \(F(\pi_1)\) is an epimorphism, that is,
  \(\pi_1\) is \(\Ideal\)\nb-epic.

  Thus the maps~\(\pi_1\) and \(\delta_n\) for
  \(n\in\N_{\ge2}\) form an \(\Ideal\)\nb-projective resolution
  of~\(N_1\).  We may now repeat the above process and
  recursively construct the phantom tower.  Thus any
  \(\Ideal\)\nb-projective resolution embeds in a phantom
  tower.  Furthermore, since the exact triangle containing a
  given morphism is unique up to isomorphism and the
  liftings~\(\pi_1\) above are unique, there is, up to
  isomorphism, only one phantom tower that contains a given
  \(\Ideal\)\nb-projective resolution.  Of course, different
  resolutions yield different phantom towers.

  Finally, it remains to lift a morphism \(f\colon A\to A'\) to
  a transformation between two given phantom towers.  First we
  can lift~\(f\) to a chain map between the
  \(\Ideal\)\nb-projective resolutions contained in these
  towers (see~\cite{Meyer-Nest:Homology_in_KK}); let
  \(P_n(f)\colon P_n\to P_n'\) for \(n\in\N\) be this chain
  map.  It remains to construct maps \(N_n(f)\colon N_n\to
  N_n'\) that together with the maps \(P_n(f)\) intertwine the
  various maps in the phantom towers.  We already have the map
  \(N_0(f)=f\).  The triangulated category axioms provide a map
  \(N_1(f)\colon N_1\to N_1'\) making the diagram
  \[\xymatrix{
    P_0\ar[d]^{P_0(f)}\ar[r]^{\pi_0}&
    A\ar[d]^{f}\ar[r]^{\iota_0^1}&
    N_1\ar[d]^{N_1(f)}\ar[r]^{\epsilon_0}&
    P_0[1]\ar[d]^{P_0(f)[1]}\\
    P_0'\ar[r]^{\pi'_0}&
    A'\ar[r]^{{\iota'}_0^1}&
    N_1'\ar[r]^{\epsilon'_0}&
    P_0'[1]
  }
  \]
  commute.  We claim that \(N_1(f)\circ\pi_1 = \pi'_1\circ
  P_1(f)\).  As above, we get short exact sequences
  \[
  \Tri_{*+1}(P_1,N_1') \into \Tri_*(P_1,P_0')
  \prto \Tri_*(P_1,A').
  \]
  Hence it suffices to check \(\epsilon'_0 \circ
  N_1(f)\circ\pi_1 = \epsilon'_0\circ \pi'_1\circ P_1(f) =
  \delta_1'\circ P_1(f)\).  But this is true because
  \(\epsilon'_0\circ N_1(f) = P_0(f)\circ\epsilon_0\) and the
  maps \(P_n(f)\) form a chain map.  Thus the map \(N_1(f)\)
  has all required properties.  Iterating this construction, we
  get the maps \(N_n(f)\) for all \(n\in\N\).  By the way, they
  need not be unique even if the maps \(P_n(f)\) are fixed.
\end{proof}

The following definition formalises an important property of
the maps~\(\iota_n^{n+1}\) in a phantom tower.

\begin{definition}
  \label{def:ideal_versal}
  Let \(\Ideal\subseteq\Tri\) be an ideal.  Let
  \(A,B\inOb\Tri\).  We call \(f\in\Ideal(A,B)\)
  \emph{\(\Ideal\)\nb-versal} if, for any \(C\inOb\Tri\), any
  \(g\in\Ideal(A,C)\) factors as \(g=h\circ f\) for some
  \(h\in\Tri(B,C)\):
  \[\xymatrix{
    A \ar[r]^{f} \ar[dr]_{g} & B \ar@{.>}[d]^{h}_{\exists} \\ & C
  }
  \]
  We do \emph{not} require this factorisation to be unique.
\end{definition}

Since~\(\Ideal\) is an ideal, any map of the form \(h\circ f\)
belongs to~\(\Ideal\).

\begin{lemma}
  \label{lem:versal_resolution}
  The maps~\(\iota_n^{n+1}\) in a phantom tower are
  \(\Ideal\)\nb-versal for all \(n\in\N\).
\end{lemma}

\begin{proof}
  Let \(f\in\Ideal(N_n,B)\).  Since~\(P_n\) is
  \(\Ideal\)\nb-projective, \(\Ideal_*(P_n,B)=0\).  Thus
  \(f\circ\pi_n=0\).  This forces~\(f\) to factor
  through~\(\iota_n^{n+1}\) because \(\Tri_*(\blank,B)\) is
  cohomological.
\end{proof}

\begin{lemma}
  \label{lem:versal_compose}
  Let \(\Ideal_1\) and~\(\Ideal_2\) be ideals in a triangulated
  category.  If \(f_1\in\Ideal_1(B,C)\) and
  \(f_2\in\Ideal_2(A,B)\) are \(\Ideal_1\)- and
  \(\Ideal_2\)\nb-versal maps, respectively, then \(f_1\circ
  f_2\colon A\to C\) is \(\Ideal_1\circ\Ideal_2\)-versal.
\end{lemma}

\begin{proof}
  Let \(h\in\Ideal_1\circ\Ideal_2(A,D)\), write \(h=h_1\circ
  h_2\) with \(h_1\in\Ideal_1\) and \(h_2\in\Ideal_2\).  Using
  versality of \(f_1\) and~\(f_2\), we find the maps \(h_2'\)
  and~\(h'\) in the following diagram:
  \[\xymatrix{
    A \ar[r]^{f_2} \ar[dr]_{h_2}&
    B \ar[r]^{f_1} \ar@{.>}[d]^{h_2'}_{\exists}&
    C  \ar@{.>}[d]^{h'}_{\exists}\\&
    \bullet \ar[r]_{h_1}&
    D.
  }
  \]
  Thus~\(h\) factors through \(f_1\circ f_2\) as required.
\end{proof}

As a consequence, the maps
\[
\iota_n^{n+k}\defeq
\iota_{n+k-1}^{n+k}\circ\dotsb\circ\iota_n^{n+1}\colon
N_n\to N_{n+k}
\]
in a phantom tower are \(\Ideal^k\)\nb-versal for all
\(n,k\in\N\).

\begin{lemma}
  \label{lem:ideal_versal}
  A map \(f\colon A\to B\) is \(\Ideal^k\)\nb-versal if and
  only if
  \[
  \Ideal^k(A,C) = \range \bigl(f^*\colon
  \Tri(B,C)\to\Tri(A,C)\bigr)
  \]
  for all \(C\inOb\Tri\).

  Let \(f\colon A\to B\) be \(\Ideal^k\)\nb-versal.  If
  \(F\colon \Tri\to\Ab\) is homological, then
  \[
  F:\Ideal^k(A) = \ker \bigl(f_*\colon F(A)\to F(B)\bigr);
  \]
  if \(\contra\colon \Tri^\op\to\Ab\) is cohomological, then
  \[
  \Ideal^k \contra(A)
  = \range\bigl(f^*\colon \contra(B)\to \contra(A)\bigr).
  \]
\end{lemma}

\begin{proof}
  This follows immediately from the definitions.
\end{proof}

As a consequence, we can compute the filtrations
\(F:\Ideal^k(A)\) and \(\Ideal^k\contra(A)\) of
Section~\ref{sec:phantom_filtrations} from the phantom tower.

\subsection{The phantom castle}
\label{sec:phantom_castle}

Now we extend the phantom tower to the phantom castle, which
contains among other things the cellular approximation tower.
We start with a phantom tower over some object \(A\inOb\Tri\).
Let
\[
\iota^n \defeq \iota_{n-1}^n \circ \dotsb \circ \iota_1^0\colon
A=N_0\to N_n
\]
and embed~\(\iota^n\) in an exact triangle
\begin{equation}
  \label{eq:tri_AAN}
  \tilde{A}_n \xrightarrow{\alpha_n}
  A \xrightarrow{\iota^n}
  N_n \xrightarrow{\gamma_n}
  \tilde{A}_n[1].
\end{equation}
The octahedral axiom relates the mapping cones
\(\tilde{A}_{n+1}\), \(P_n\), and \(\tilde{A}_n\) of the maps
\(\iota^{n+1}\), \(\iota_n^{n+1}\), and~\(\iota^n\) because
\(\iota^{n+1}=\iota_n^{n+1}\circ\iota^n\) (see
\cite{Neeman:Triangulated}*{Proposition I.4.6} or
\cite{Meyer-Nest:BC}*{Proposition A.1}).  More precisely, the
octahedral axiom allows us to choose maps
\begin{equation}
  \label{eq:tri_AAP}
  \tilde{A}_n \xrightarrow{\alpha_n^{n+1}}
  \tilde{A}_{n+1} \xrightarrow{\sigma_n}
  P_n \xrightarrow{\kappa_n}
  \tilde{A}_n[1],
\end{equation}
such that this triangle is exact and the following diagram
commutes:
\begin{equation}
  \label{eq:octahedral_relations}
  \begin{gathered}
    \xymatrix{
      N_n \ar[r]^{\iota_n^{n+1}} \ar[d]|-\circ_{\gamma_n}&
      N_{n+1} \ar[dr]|-\circ^{\epsilon_n}
      \ar[d]|-\circ_{\gamma_{n+1}}\\
      \tilde{A}_n \ar[r]^{\alpha_n^{n+1}} \ar[dr]_{\alpha_n}&
      \tilde{A}_{n+1} \ar[d]^{\alpha_{n+1}} \ar[r]^{\sigma_n}&
      P_n\ar[d]_{\pi_n}\ar[r]|-\circ^{\kappa_n}&
      \tilde{A}_n\\&
      A\ar[r]^{\iota^n}\ar[dr]_{\iota^{n+1}}&
      N_n\ar[ur]|-\circ_{\gamma_n}\ar[d]^{\iota_n^{n+1}}\\
      &&N_{n+1}
    }
  \end{gathered}
\end{equation}
In addition, we can achieve that the triangle
\begin{equation}
  \label{eq:htpy_pull-back}
  N_n[-1] \to
  \tilde{A}_{n+1} \xrightarrow{\bigl(
    \begin{smallmatrix}\alpha_{n+1}\\\sigma_n\end{smallmatrix}
    \bigr)}
  A\oplus P_n \xrightarrow{(\iota^n,\pi_n)} N_n
\end{equation}
is exact, that is, the square in the middle
of~\eqref{eq:octahedral_relations} is homotopy Cartesian and
the diagonal of the top right square provides its differential.

\begin{lemma}
  \label{lem:approx_tower_projective}
  The object~\(\tilde{A}_n\) is \(\Ideal^n\)\nb-projective for
  each \(n\in\N\).
\end{lemma}

\begin{proof}[Proof by induction on~\(n\)] The case \(n=0\) is
  clear.  Since \(P_n\inOb\Proj_\Ideal\) for all \(n\in\N\),
  the exact triangles~\eqref{eq:tri_AAP} and
  Proposition~\ref{pro:product_class} provide the induction
  step.
\end{proof}

Furthermore, the map \(\alpha_n\colon \tilde{A}_n\to A\) is
\(\Ideal^n\)\nb-epic because \(\iota^n\in\Ideal^n\), so that it
is the first step of an \(\Ideal^n\)\nb-projective resolution
of~\(A\).  This provides another explanation why the
map~\(\iota^n\) is \(\Ideal\)\nb-versal (compare
Lemma~\ref{lem:versal_resolution}).

\begin{remark}
  \label{ref:sparse_phantom_tower}
  The cone of the map \(\iota_m^{m+k}\) is
  \(\Ideal^k\)\nb-projective for all \(m,k\in\N\) by a similar
  argument.  Hence
  \[
  A=N_0
  \xrightarrow{\iota_0^k} N_k
  \xrightarrow{\iota_k^{2k}} N_{2k}
  \xrightarrow{\iota_{2k}^{3k}} N_{3k}
  \to \dotsb
  \]
  together with the exact triangles that contain the
  maps~\(\iota_{jk}^{jk+k}\) is an \(\Ideal^k\)\nb-phantom
  tower and hence yields an \(\Ideal^k\)\nb-projective
  resolution by Lemma~\ref{lem:phantom_tower}.

  As a result, an \(\Ideal\)\nb-phantom tower determines
  \(\Ideal^k\)\nb-phantom towers for all \(k\in\N\).
\end{remark}

\begin{definition}
  \label{def:approx_tower}
  The sequence of exact triangles~\eqref{eq:tri_AAP} is called
  the \emph{cellular approximation tower} over~\(A\).
\end{definition}

The motivation for our terminology is the following.  If~\(A\)
has a projective resolution of finite length, then we can
choose a phantom tower with \(P_n=0\) for \(n\gg0\).  Suppose,
in addition, that~\(A\) belongs to the thick triangulated
subcategory generated by~\(\Proj_\Ideal\).  Then the proof of
Proposition~\ref{pro:Ct_for_projective} yields \(N_n=0\) for
\(n\gg0\).  The exact triangles~\eqref{eq:tri_AAN} mean that
the maps \(\alpha_n\colon \tilde{A}_n\to A\) become invertible
for \(n\gg0\), that is, \(\tilde{A}_n\cong A\).  Therefore, we
think of the objects~\(N_n\) as ``obstructions'' that should
get smaller for \(n\to\infty\), and of the
objects~\(\tilde{A}_n\) as better and better approximations
to~\(A\).  They are called ``\(\Proj_\Ideal\)\nb-cellular''
because they are constructed out of \(\Ideal\)\nb-projective
objects -- the cells -- by iterated exact triangles.

\begin{definition}
  \label{def:phantom_castle}
  A \emph{phantom castle} over~\(A\) is a sequence of objects
  \(N_n\), \(P_n\), \(\tilde{A}_n\) with maps
  \(\iota_n^{n+1}\), \(\pi_n\), \(\epsilon_n\), \(\iota^n\),
  \(\alpha_n\), \(\gamma_n\), \(\sigma_n\), \(\kappa_n\) such
  that the triangles \eqref{eq:tri_PNN}, \eqref{eq:tri_AAN},
  \eqref{eq:tri_AAP}, and~\eqref{eq:htpy_pull-back} are exact
  and the diagram~\eqref{eq:octahedral_relations} commutes.
\end{definition}

We will use most of the information encoded in this definition
to identify the spectral sequences generated by the phantom
tower and the cellular approximation tower; only the
commutativity of the square in the middle
of~\eqref{eq:octahedral_relations} and the exact
sequence~\eqref{eq:htpy_pull-back} seem irrelevant in the
following.

\subsection{Complementary pairs of subcategories and
  localisation}
\label{sec:complementary}

We call two thick subcategories \(\Local\) and~\(\Null\)
of~\(\Tri\) \emph{complementary} if \(\Tri_*(L,N)=0\) for all
\(L\inOb\Local\), \(N\inOb\Null\) and, for any \(A\inOb\Tri\),
there is an exact triangle \(L\to A\to N\to L[1]\) with
\(L\inOb\Local\) and \(N\inOb\Null\) (see
\cite{Meyer-Nest:BC}*{Definition 2.8}).  Similar situations
have been studied by various authors, under various names, such
as localisation pairs, stable t\nb-structures, torsion pairs; a
complementary pair is equivalent to a localisation
functor~\(L\) on~\(\Tri\), where~\(\Local\) is the class of
\(L\)\nb-local objects and~\(\Null\) is the class of
\(L\)\nb-acyclic objects.

The following assertions are contained in
\cite{Meyer-Nest:BC}*{Proposition 2.9}.  Let \((\Local,\Null)\)
be complementary.  Then the exact triangle \(L\to A\to N\to
L[1]\) with \(L\inOb\Local\) and \(N\inOb\Null\) is unique and
functorial, and the resulting functors \(L\colon
\Tri\to\Local\) and \(N\colon \Tri\to\Null\) mapping~\(A\) to
\(L\) and~\(N\), respectively, are left adjoint to the
embedding functor \(\Local\to\Tri\) and right adjoint to the
embedding functor \(\Null\to\Tri\), respectively.  That is, the
subcategory~\(\Null\) is reflective and~\(\Local\) is
coreflective.  The composite functors
\(\Local\to\Tri\to\Tri/\Null\) and
\(\Null\to\Tri\to\Tri/\Local\) are equivalences of categories.

Conversely, let \(\Null\subseteq\Tri\) be a reflective
subcategory and let \(N\colon \Tri\to\Null\) be the left
adjoint of the embedding functor \(\Null\to\Tri\).  Let
\[
\Local = \{A\inOb\Tri\mid N(A)=0\}
\]
be the \emph{left orthogonal complement} of~\(\Null\).  Then
\((\Local,\Null)\) is a complementary pair of subcategories,
and~\(\Local\) is the only possible partner for~\(\Null\).
Thus complementary pairs are essentially the same as reflective
subcategories.  Dually, a subcategory~\(\Local\) is
coreflective if and only if it is part of a complementary pair
\((\Local,\Null)\), and the only candidate for~\(\Null\) is the
right orthogonal complement of~\(\Local\).

If \(F\colon \Tri\to\Cat\) is a covariant functor, then its
\emph{localisation} \(\Left F\) with respect to~\(\Null\) is
defined by \(\Left F\defeq F\circ L\), where \(L\colon
\Tri\to\Local\) is the right adjoint of the embedding
\(\Local\to\Tri\).  The natural maps \(L(A)\to A\) provide a
natural transformation \(\Left F\Rightarrow F\).  If
\(\contra\colon \Tri^\op\to\Cat\) is a contravariant functor,
then the localisation \(\contra\circ L\) is denoted by~\(\Right
\contra\).  It comes together with a natural transformation
\(\contra\Rightarrow \Right \contra\).

This localisation process is an important tool to construct
functors.  Special cases are derived functors in homological
algebra and the domain of the Baum--Connes assembly map
(see~\cite{Meyer-Nest:BC}).

Although the definition of a complementary pair is symmetric,
the subcategories \(\Local\) and~\(\Null\) have a rather
different nature in most examples.  Usually, one of them -- here
it is always~\(\Null\) -- is defined directly and the other one
is only described by generators.  This makes it hard to tell
which objects it contains and to find the exact triangles
needed for complementarity.

Here homological ideals help.  Let~\(\Ideal\) be a homological
ideal with enough projective objects in a triangulated
category~\(\Tri\).  Let~\(\gen{\Proj_\Ideal}\) be the
localising subcategory generated by~\(\Proj_\Ideal\), that is,
the smallest localising subcategory of~\(\Tri\) that
contains~\(\Proj_\Ideal\).  Since the name ``projective'' is
already taken, we call objects of~\(\gen{\Proj_\Ideal}\)
\emph{\(\Proj_\Ideal\)\nb-cellular}.  We have
\(\Proj_\Ideal^n\subseteq \gen{\Proj_\Ideal}\) for all
\(n\in\N\cup\{\infty,2\infty,\dotsc\}\) by Propositions
\ref{pro:product_class} and~\ref{pro:intersect_class}.

\begin{definition}
  \label{def:Null_Ideal}
  Let~\(\Null_\Ideal\) be the full subcategory of
  \(\Ideal\)\nb-contractible objects, that is, objects~\(N\)
  with \(\ID_N\in\Ideal(N,N)\).
\end{definition}

An object~\(N\) is \(\Ideal\)\nb-contractible if and only if
\(0\to N\) is an \(\Ideal\)\nb-projective resolution.  Thus all
\(\Ideal\)\nb-derived functors vanish on~\(\Null_\Ideal\).

Now the following question arises: is the pair of subcategories
\((\gen{\Proj_\Ideal},\Null_\Ideal)\) complementary?  It is
evident that \(\Tri(P,N)=0\) if \(P\in\Proj_\Ideal\) and
\(N\in\Null_\Ideal\).  This extends to
\(P\in\gen{\Proj_\Ideal}\) because the left orthogonal
complement of~\(\Null_\Ideal\) is localising.  This is the easy
half of the definition of a complementary pair.  The other,
non-trivial half requires an additional condition on the
ideal~\(\Ideal\).

\subsection{Compatibility with direct sums}
\label{sec:compatible_sums}

\begin{definition}
  \label{def:ideal_sums}
  An ideal~\(\Ideal\) is called \emph{compatible with countable
    direct sums} if, for any countable family~\((A_i)_{i\in
    I}\) of objects of~\(\Tri\), the canonical isomorphism
  \[
  \Tri\biggl(\bigoplus_{i\in I} A_i,B\biggr)
  \cong \prod_{i\in I} \Tri(A_i,B)
  \]
  restricts to an isomorphism \(\Ideal\bigl(\bigoplus_{i\in I}
  A_i,B\bigr) \cong \prod_{i\in I} \Ideal(A_i,B)\).
\end{definition}

An ideal~\(\Ideal\) is compatible with countable direct sums if
and only if the following holds: given countable families of
objects \((A_i)\), \((B_i)\) and maps \(f_i\in\Ideal(A_i,B_i)\)
for \(i\in I\), we have \(\bigoplus f_i \in
\Ideal\bigl(\bigoplus A_i,\bigoplus B_i\bigr)\).

Recall that direct sums of exact triangles are again exact
(see~\cite{Neeman:Triangulated}) and that the ideal determines
and is determined by the classes of \(\Ideal\)\nb-epimorphisms
or of \(\Ideal\)\nb-monomorphisms.  Therefore, \(\Ideal\) is
compatible with direct sums if and only if direct sums of
\(\Ideal\)\nb-monomorphisms are again
\(\Ideal\)\nb-monomorphisms, if and only if direct sums of
\(\Ideal\)\nb-epimorphisms are again
\(\Ideal\)\nb-epimorphisms.

Moreover, if~\(\Ideal\) is compatible with countable direct
sums, then a direct sum of \(\Ideal\)\nb-equivalences is again
an \(\Ideal\)\nb-equivalence, and~\(\Null_\Ideal\) is a
localising subcategory of~\(\Tri\).  Moreover, a direct sum of
phantom castles over~\(A_i\) is a phantom castle over
\(\bigoplus A_i\).

\begin{example}
  \label{exa:ker_compatible_sums}
  Let~\(F\) be a stable homological functor or an exact functor
  to another triangulated category, and suppose that~\(F\)
  commutes with countable direct sums.  Then \(\ker F\) is a
  homological ideal and compatible with countable direct sums.
\end{example}

This example is, in fact, already the most general case:

\begin{proposition}
  \label{pro:ker_compatible_sum}
  Let~\(\Tri\) be a triangulated category with countable direct
  sums and let~\(\Ideal\) be a homological ideal in~\(\Tri\).
  Let \(F\colon \Tri\to\Abel_\Ideal\Tri\) be a universal
  \(\Ideal\)\nb-exact stable homological functor.  The
  ideal~\(\Ideal\) is compatible with countable direct sums if
  and only if the Abelian category \(\Abel_\Ideal\Tri\) has
  exact countable direct sums and the functor \(F\colon
  \Tri\to\Abel_\Ideal\Tri\) commutes with countable direct
  sums.
\end{proposition}

\begin{proof}
  One direction is just the assertion in
  Example~\ref{exa:ker_compatible_sums}.  The other direction
  requires some description of the universal functor~\(F\).  We
  use the description in~\cite{Meyer-Nest:Homology_in_KK},
  which starts with the homotopy category \(\Ho(\Tri)\) of
  chain complexes with entries in~\(\Tri\).  Since~\(\Tri\) has
  countable direct sums, so has \(\Ho(\Tri)\).  The
  \(\Ideal\)\nb-exact chain complexes form a thick
  subcategory~\(\mathcal{E}\) of \(\Ho(\Tri)\); it is closed
  under countable direct sums because~\(\Ideal\) is compatible
  with countable direct sums.  Hence the localisation
  \(\Ho(\Tri)/\mathcal{E}\) still has countable direct sums.

  The Abelian approximation \(\Abel_\Ideal\Tri\) is equivalent
  to the heart of a canonical truncation structure on
  \(\Ho(\Tri)/\mathcal{E}\) described
  in~\cite{Meyer-Nest:Homology_in_KK} and consists of chain
  complexes that are exact in degrees not equal to~\(0\).  The
  universal functor~\(F\) is the obvious one, viewing an object
  of~\(\Tri\) as a chain complex supported in degree~\(0\).  It
  is evident that the subcategory \(\Abel_\Ideal\Tri\subseteq
  \Ho(\Tri)/\mathcal{E}\) is closed under countable direct
  sums.  Countable direct sums of extensions in
  \(\Abel_\Ideal\Tri\) remain extensions because the analogous
  assertion holds for direct sums of exact triangles in any
  triangulated category (see~\cite{Neeman:Triangulated}) and
  extensions in the heart are related to exact triangles in the
  ambient triangulated category.  Clearly, the functor
  \(\Tri\to\Abel_\Ideal\Tri\) preserves countable direct sums.
\end{proof}

\begin{example}
  \label{exa:finite_rank}
  The ideal of finite rank operators on the category of vector
  spaces is an ideal that is not compatible with countable
  direct sums.
\end{example}

\subsection{Complementarity and structure of cellular objects}
\label{sec:cellular_objects}

The results in this section generalise results of Apostolos
Beligiannis (see \cite{Beligiannis:Relative}*{Theorem 6.5},
\cite{Beligiannis:Relative}*{Corollary 5.12}) in the case
where~\(\Proj_\Ideal\) is generated by a single compact object.

\begin{theorem}
  \label{the:ideal_sum_complementary}
  Let~\(\Tri\) be a triangulated category with countable direct
  sums, and let~\(\Ideal\) be a homological ideal in~\(\Tri\)
  with enough projective objects.  Suppose that~\(\Ideal\) is
  compatible with countable direct sums.  Then the pair of
  localising subcategories
  \((\gen{\Proj_\Ideal},\Null_\Ideal)\) in~\(\Tri\) is
  complementary.
\end{theorem}

We will present two independent proofs, one using phantom
towers, the other cellular approximation towers.  Both require
homotopy colimits:

\begin{definition}
  \label{def:homotopy_limit}
  Let \((D_n,\varphi_n^{n+1})\) be an inductive system
  in~\(\Tri\).  Define the \emph{shift}
  \[
  S\colon \bigoplus D_n\to \bigoplus D_n,
  \qquad
  S|_{D_n}\colon D_n \xrightarrow{\varphi_n^{n+1}} D_{n+1}
  \subseteq \bigoplus D_n.
  \]
  The \emph{homotopy colimit} \(\hoinjlim
  {}(D_n,\varphi_n^{n+1})\) is the third leg in the exact
  triangle
  \[\xymatrix{
    \bigoplus D_n \ar[r]^-{\ID-S} &
    \bigoplus D_n \ar[r] &
    \hoinjlim {}(D_n,\varphi_n^{n+1}) \ar[r] &
    \bigoplus D_n[1].
  }\]
  Recall that \(\ID-S\) determines this triangle uniquely up to
  isomorphism.
\end{definition}

\begin{proof}[Proof of Theorem~\ref{the:ideal_sum_complementary}]
  Since the class of \(A\inOb\Tri\) with \(\Tri_*(A,B)=0\) for
  all \(B\inOb\Null_\Ideal\) is localising, we have
  \(\Tri_*(A,B)=0\) if \(A\inOb\gen{\Proj_\Ideal}\) and
  \(B\inOb\Null_\Ideal\).  It remains to construct, for each
  \(A\inOb\Tri\), an exact triangle \(\tilde{A} \to A\to
  N\to\tilde{A}[1]\) with \(N\inOb\Null_\Ideal\) and
  \(\tilde{A}\inOb\gen{\Proj_\Ideal}\).

  Construct a phantom castle over~\(A\) and let \(N \defeq
  \hoinjlim {}(N_n,\iota_n^{n+1})\) be the homotopy colimit of
  the phantom tower.  We also use the homotopy colimit of the
  constant inductive system \((A,\ID_A)\).  This is just~\(A\)
  because of the split exact triangle
  \begin{equation}
    \label{eq:holim_constant}
    \bigoplus A \xrightarrow{\ID-S}
    \bigoplus A \xrightarrow{\nabla}
    A \xrightarrow{0}
    \bigoplus A[1],
  \end{equation}
  where~\(\nabla\) is the codiagonal map.  By
  \cite{Beilinson-Bernstein-Deligne}*{Proposition 1.1.11} (and
  a rotation), we can find~\(\tilde{A}\) and the dotted arrows
  in the following diagram
  \begin{equation}
    \label{eq:holim_diagram}
    \begin{gathered}
      \xymatrix{
        \bigoplus \tilde{A}_n \ar@{.>}[r] \ar[d]^{\bigoplus \alpha_n}&
        \bigoplus \tilde{A}_n \ar@{.>}[r] \ar[d]^{\bigoplus \alpha_n}&
        \tilde{A} \ar@{.>}[r]|-\circ \ar@{.>}[d]&
        \bigoplus \tilde{A}_n \ar[d]^{\bigoplus \alpha_n}\\
        \bigoplus A \ar[r]^{\ID-S} \ar[d]^{\bigoplus \iota^n}&
        \bigoplus A \ar[d]^{\bigoplus \iota^n} \ar[r]^{\nabla}&
        A \ar[r]|-\circ^{0} \ar@{.>}[d]&
        \bigoplus A \ar[d]^{\bigoplus \iota^n}\\
        \bigoplus N_n \ar[r]^{\ID-S} \ar[d]|-\circ^{\bigoplus \gamma_n}&
        \bigoplus N_n \ar[d]|-\circ^{\bigoplus \gamma_n} \ar[r]&
        N \ar[r]|-\circ \ar@{.>}[d]|-\circ \ar@{}[dr]|{-}&
        \bigoplus N_n \ar[d]|-\circ^{-\bigoplus \gamma_n}\\
        \bigoplus \tilde{A}_n \ar@{.>}[r] &
        \bigoplus \tilde{A}_n \ar@{.>}[r] &
        \tilde{A}[1] \ar@{.>}[r]|-\circ &
        \bigoplus \tilde{A}_n
      }
    \end{gathered}
  \end{equation}
  so that the rows and columns are exact triangles and the
  squares commute except for the one marked with a minus sign,
  which anti-commutes.

  Lemma~\ref{lem:approx_tower_projective} yields
  \(\tilde{A}_n\inOb\Proj_\Ideal^n\) for all \(n\in\N\).  Hence
  \(\bigoplus \tilde{A}_n\inOb \Proj_\Ideal^\infty\subseteq
  \gen{\Proj_\Ideal}\) by
  Proposition~\ref{pro:intersect_class}.  The exactness of the
  first row in~\eqref{eq:holim_diagram} implies
  \(\tilde{A}\inOb\gen{\Proj_\Ideal}\).  We claim that
  \(N\inOb\Null_\Ideal\).  Hence the third column
  in~\eqref{eq:holim_diagram} is the kind of exact triangle we
  need for \(\gen{\Proj_\Ideal}\) and~\(\Null_\Ideal\) to be
  complementary.

  Let~\(F\) be a stable homological functor with \(\ker
  F=\Ideal\).  We must show \(F(N)=0\).  The map~\(S\) factors
  through \(\bigoplus \iota_n^{n+1}\); this map belongs to
  \(\Ideal=\ker F\) because~\(\Ideal\) is compatible with
  direct sums.  Hence \(F(\ID-S)=F(\ID)\) is invertible.  By a
  long exact sequence, this implies \(F(N)=0\), that is,
  \(N\inOb\Null_\Ideal\).
\end{proof}

Suppose from now on that we are in the situation of
Theorem~\ref{the:ideal_sum_complementary}.  Since
\(\gen{\Proj_\Ideal}\) and~\(\Null_\Ideal\) are complementary,
there is a unique exact triangle \(\tilde{A} \to A \to
N\to\tilde{A}[1]\) with \(\Ideal\)\nb-contractible~\(N\) and
\(\Proj_\Ideal\)\nb-cellular~\(\tilde{A}\); we call
\(\tilde{A}\) the \emph{\(\Proj_\Ideal\)\nb-cellular
  approximation} of~\(A\).  Even more, \(\tilde{A}\) and~\(N\)
depend functorially on~\(A\), so that we get two functors
\(L\colon \Tri\to\gen{\Proj_\Ideal}\) and \(N\colon
\Tri\to\Null_\Ideal\).

The proof of Theorem~\ref{the:ideal_sum_complementary} above
provides an explicit model for \(N(A)\): it is the homotopy
colimit of the phantom tower of~\(A\).

\begin{proposition}
  \label{pro:cellular_approx}
  Let \(A\inOb\Tri\) and construct a phantom castle over~\(A\).
  Then \(L(A) = \tilde{A}\) is the homotopy colimit of the
  cellular approximation tower \((\tilde{A}_n)_{n\in\N}\).
\end{proposition}

This does not yet follow from~\eqref{eq:holim_diagram} because
we cannot control the dotted maps.

\begin{proof}
  Let \(\tilde{A}\defeq \hoinjlim
  {}(\tilde{A}_n,\alpha_n^{n+1})\).  We compare the exact
  triangle that defines the homotopy colimit~\(\tilde{A}\) with
  the triangle~\eqref{eq:holim_constant}.  The triangulated
  category axioms provide \(f\in\Tri(\tilde{A},A)\) that makes
  the following diagram commute:
  \begin{equation}
    \label{eq:holim_induce}
    \begin{gathered}
      \xymatrix{
        \bigoplus \tilde{A}_n
        \ar[r]^{\ID-S} \ar[d]^{\bigoplus \alpha_n}&
        \bigoplus \tilde{A}_n
        \ar[r] \ar[d]^{\bigoplus \alpha_n}&
        \tilde{A} \ar@{.>}[d]^{f} \ar[r]|-\circ &
        \bigoplus \tilde{A}_n \ar[d]^{\bigoplus \alpha_n}\\
        \bigoplus A \ar[r]^{\ID-S} &
        \bigoplus A \ar[r]^{\nabla} &
        A \ar[r]|-\circ^-{0} &
        \bigoplus A.
      }
    \end{gathered}
  \end{equation}

  We claim that~\(f\) is an \(\Ideal\)\nb-equivalence.
  Equivalently, the cone of~\(f\) is
  \(\Ideal\)\nb-contractible, so that the mapping cone triangle
  for~\(f\) has entries in \(\gen{\Proj_\Ideal}\)
  and~\(\Null_\Ideal\); this implies that \(L(A)=\tilde{A}\).

  Let~\(F\) be a stable homological functor with \(\ker
  F=\Ideal\).  We check that \(F(f)\) is invertible.  The
  direct sum of the triangles~\eqref{eq:tri_AAN} for \(n\in\N\)
  is again an exact triangle.  On the long exact homology
  sequence
  \begin{multline*}
    \dotsb
    \longrightarrow F_{m+1}\Bigl(\bigoplus N_n\Bigr)
    \\ \longrightarrow F_m\Bigl(\bigoplus \tilde{A}_n\Bigr)
    \longrightarrow F_m\Bigl(\bigoplus A\Bigr)
    \longrightarrow F_m\Bigl(\bigoplus N_n\Bigr)
    \longrightarrow \dotsb
  \end{multline*}
  for this exact triangle, consider the operator induced by
  \(\ID-S\) on each entry.  Since~\(\Ideal\) is compatible with
  direct sums, the shift map~\(S\) on \(\bigoplus N_n\) is a
  phantom map, so that \(F(\ID-S)\) acts identically on
  \(F\bigl(\bigoplus N_n\bigr)\).  On \(F\bigl(\bigoplus
  A\bigr)\), the map \(\ID-S\) induces a split monomorphism
  with cokernel \(F(\nabla)\colon F\bigl(\bigoplus A\bigr)\to
  F(A)\).  Now a diagram chase shows that the map on
  \(F\bigl(\bigoplus \tilde{A}_n\bigr)\) induced by \(\ID-S\)
  is injective and has the cokernel
  \[
  F\Bigl(\bigoplus \tilde{A}_n\Bigr)
  \xrightarrow{F\bigl(\bigoplus \alpha_n\bigr)}
  F\Bigl(\bigoplus A\Bigr) \xrightarrow{F(\nabla)}
  F(A).
  \]
  Comparing the long exact homology sequences for the two rows
  in~\eqref{eq:holim_induce}, we conclude that \(F(f)\) is
  indeed invertible.
\end{proof}

We have proved Proposition~\ref{pro:cellular_approx} by
constructing an \(\Ideal\)\nb-equivalence between an arbitrary
object of~\(\Tri\) and the homotopy colimit of its cellular
approximation tower.  This provides another, independent proof
of Theorem~\ref{the:ideal_sum_complementary}.  Since the exact
triangle \(\tilde{A}\to A\to N\to \tilde{A}[1]\) with
\(\tilde{A}\inOb\gen{\Proj_\Ideal}\) and \(N\inOb\Null_\Ideal\)
is unique up to isomorphism, both proofs construct the same
exact triangle.  The first proof shows that~\(N\) is the
homotopy colimit of the phantom tower, the second one shows
that~\(\tilde{A}\) is the homotopy colimit of the cellular
approximation tower.  We conclude, therefore, that the
object~\(\tilde{A}\) in~\eqref{eq:holim_diagram} is the
homotopy colimit of the phantom tower and that the map
\(\tilde{A}\to A\) in~\eqref{eq:holim_diagram} agrees with the
map~\(f\) from~\eqref{eq:holim_induce}.

\begin{theorem}
  \label{the:cellular_limit_tower}
  Let~\(\Tri\) be a triangulated category with countable direct
  sums, and let~\(\Ideal\) be a homological ideal in~\(\Tri\)
  that is compatible with countable direct sums and has enough
  projective objects.  Let \(A\inOb\Tri\) and construct a
  phantom castle over~\(A\).  The following are equivalent:
  \begin{enumerate}[label=\textup{(\arabic{*})}]
  \item \(A\) is \(\Proj_\Ideal\)\nb-cellular, that is,
    \(A\inOb\gen{\Proj_\Ideal}\);

  \item \(A\) is isomorphic to the homotopy colimit of its
    cellular approximation tower;

  \item \(A\) is isomorphic to the homotopy colimit of an
    inductive system \((P_n,\varphi_n)\) with \(P_n\inOb
    \bigcup_{k\in\N} \Proj_\Ideal^k\) for all \(n\in\N\);

  \item \(A\) is \(\Ideal^{2\infty}\)\nb-projective.
  \end{enumerate}
  As a consequence, \(\Ideal^{2\infty}=\Ideal^{n\infty}\) for
  all \(n\ge2\).
\end{theorem}

\begin{proof}
  Since \(L(A)\cong A\) if and only if
  \(A\inOb\gen{\Proj_\Ideal}\),
  Proposition~\ref{pro:cellular_approx} yields the equivalence
  of~(1) and~(2).  The implication (2)\(\Longrightarrow\)(3) is
  trivial: the cellular approximation tower provides an
  inductive system of the required kind.

  We check that~(3) implies~(4).  Let \((P_n,\varphi_n)\) be an
  inductive system as in~(3).  First,
  Proposition~\ref{pro:intersect_class} shows that \(\bigoplus
  P_n\) is \(\Ideal^\infty\)\nb-projective.  Then
  Proposition~\ref{pro:product_class} shows that the homotopy
  colimit is \(\Ideal^{2\infty}\)\nb-projective.

  Propositions \ref{pro:intersect_class}
  and~\ref{pro:product_class} show recursively that all
  \(\Ideal^\alpha\)\nb-projective objects belong
  to~\(\gen{\Proj_\Ideal}\) for
  \(n=0,1,2,3,\dotsc,\infty,2\cdot\infty\).  Hence~(4)
  implies~(1), so that all four conditions are equivalent.

  Finally, since \(\Proj_\Ideal^{2\infty}=\gen{\Proj_\Ideal}\),
  it follows from Proposition~\ref{pro:product_class} that the
  powers \(\Ideal^{n\infty}\) for \(n\ge2\) have the same
  projective objects.  Therefore, they are all equal.
\end{proof}

\begin{proposition}
  \label{pro:localise_phantom_castle}
  The \(\Proj_\Ideal\)\nb-cellular approximation functor
  \(L\colon \Tri\to\gen{\Proj_\Ideal}\) maps a phantom castle
  over \(A\inOb\Tri\) to a phantom castle over \(L(A)\).  A
  morphism \(f\in\Tri(A,B)\) belongs to~\(\Ideal^\alpha\) for
  some~\(\alpha\) if and only if
  \(L(f)\in\Tri\bigl(L(A),L(B)\bigr)\) does.
\end{proposition}

\begin{proof}
  Since~\(L\) is an exact functor, it preserves the commuting
  diagrams and exact triangles required for a phantom castle.
  It also maps~\(\Proj_\Ideal\) to itself because \(L(B)\cong
  B\) for all \(B\inOb\gen{\Proj_\Ideal}\).  It remains to
  check that \(L(f)\in\Ideal^\alpha\bigl(L(B),L(B')\bigr)\) if
  and only if \(f\in\Ideal^\alpha(B,B')\).  Let~\(F\) be a
  stable homological functor with \(\Ideal^\alpha = \ker F\).
  Then \(F(N)=0\) for all \(N\inOb\Null_\Ideal\).  Therefore,
  \(F\) descends to the localisation \(\Tri/\Null_\Ideal\);
  equivalently, the natural transformation \(F\circ
  L\Rightarrow F\) is an isomorphism.  In particular,
  \(F(f)=0\) if and only if \(F\bigl(L(f)\bigr)=0\).
\end{proof}

As a result, it makes almost no difference whether we work in
\(\Tri\) or \(\Tri/\Null_\Ideal\).  We work in~\(\Tri\) most of
the time and allow~\(\Null_\Ideal\) to be non-trivial in order
to formulate Theorem~\ref{the:ideal_sum_complementary}.

The direct sums \(\bigoplus \tilde{A}_n\)
in~\eqref{eq:holim_induce} are \(\Ideal^\infty\)\nb-projective
by Proposition~\ref{pro:intersect_class}.  The map
\(\nabla\circ \bigoplus \alpha_n\colon \bigoplus \tilde{A}_n\to
A\) is \(\Ideal^\infty\)\nb-epic because it is
\(\Ideal^n\)\nb-epic for all \(n\in\N\).  We may replace~\(A\)
by~\(\tilde{A}\) in this statement by
Proposition~\ref{pro:localise_phantom_castle}.  Thus the top row
in~\eqref{eq:holim_induce} is an \(\Ideal^\infty\)\nb-exact
triangle.  This means that the chain complex
\[
\dotsb \to 0 \to
\bigoplus \tilde{A}_n \xrightarrow{\ID-S}
\bigoplus \tilde{A}_n \to \tilde{A}
\]
is \(\Ideal^\infty\)\nb-exact and hence an
\(\Ideal^\infty\)\nb-projective resolution of~\(\tilde{A}\).
Once again, Proposition~\ref{pro:localise_phantom_castle}
allows us to replace~\(\tilde{A}\) by~\(A\) in this statement,
that is, we get an \(\Ideal^\infty\)\nb-projective resolution
of length~\(1\)
\begin{equation}
  \label{eq:Ideal_infty_resolution}
  \dotsb \to 0 \to
  \bigoplus \tilde{A}_n \xrightarrow{\ID-S}
  \bigoplus \tilde{A}_n \to A.
\end{equation}
This will allow us to analyse the convergence of the ABC
spectral sequence.

\subsection{Complementarity via partially defined adjoints}
\label{sec:kern_adjointable}

Suppose that we are given a thick subcategory~\(\Null\) of a
triangulated category~\(\Tri\) and that we want to use
Theorem~\ref{the:ideal_sum_complementary} to show that it is
reflective, that is, there is another thick
subcategory~\(\Local\) such that \((\Local,\Null)\) is
complementary.

To have a chance of doing so, \(\Tri\) must have countable
direct sums, and the subcategory~\(\Null\) must be localising,
that is, closed under countable direct sums: this happens
whenever Theorem~\ref{the:ideal_sum_complementary} applies.  By
\cite{Meyer-Nest:BC}*{Proposition 2.9}, the only candidate
for~\(\Local\) is the left orthogonal complement
\[
\Local \defeq \{A\inOb\Tri\mid
\text{\(\Tri(A,N)=0\) for all \(N\inOb\Null\)}\}
\]
of~\(\Null\), which is another localising subcategory.

The starting point of our method is a stable additive functor
\(F\colon \Tri\to\Addi\) with
\[
\Null = \Null_F \defeq \{A\inOb\Tri\mid F(A)=0\}.
\]
This functor yields a stable ideal \(\Ideal_F\defeq \ker F\).
In applications, \(F\) is either a stable homological functor
to a stable Abelian category or an exact functor to another
triangulated category; in either case, the ideal \(\ker F\) is
homological and~\(\Null_F\) is the class of all
\(\Ideal_F\)\nb-contractible objects.  In addition, we
assume~\(F\) to \emph{commute with countable direct sums}, so
that~\(\Ideal_F\) is compatible with countable direct sums.

In order to apply Theorem~\ref{the:ideal_sum_complementary}, it
remains to prove that there are enough
\(\Ideal_F\)\nb-projective objects in~\(\Tri\).  Then the pair
of subcategories \((\gen{\Proj_{\Ideal_F}}, \Null)\) is
complementary.  For a \emph{good} choice of~\(F\), this may be
much easier than proving directly that \((\Local,\Null)\) is
complementary.  The choice in the following example never
helps.  But often, there is another choice for~\(F\) that does.

\begin{example}
  \label{exa:bad_functor}
  The localisation functor \(\Tri\to\Tri/\Null\) is a possible
  choice for~\(F\) -- that is, it has the right kernel on
  objects -- but it should be a bad one because it tells us
  nothing new about~\(\Null\).  In fact, the \(\ker
  F\)\nb-projective objects are precisely the objects
  of~\(\Local\), so that we have not gained anything.
\end{example}

We now discuss a sufficient condition for enough projective
objects from~\cite{Meyer-Nest:Homology_in_KK}.  The \emph{left
  adjoint} of \(F\colon \Tri\to\Addi\) is defined on an object
\(A\inOb\Addi\) if the functor \(B\mapsto
\Addi\bigl(A,F(B)\bigr)\) on~\(\Tri\) is representable, that
is, there is an object~\(F^\lad(A)\) of~\(\Tri\) and a
natural isomorphism \(\Tri(F^\lad(A),B)\cong
\Addi\bigl(A,F(B)\bigr)\) for all \(B\inOb\Tri\).  We say
that~\(F^\lad\) is \emph{defined on enough objects} if, for
any object~\(B\) of~\(\Addi\) there is an epimorphism \(B'\to
B\) such that~\(F^\lad\) is defined on~\(B'\).

The following theorem asserts that~\(\Ideal_F\) has enough
projective objects if~\(F^\lad\) is defined on enough
objects.  The statement is somewhat more involved because it is
often useful to shrink the domain of definition
of~\(F^\lad\) to a sufficiently big
subcategory~\(\Proj\Addi\).

\begin{theorem}
  \label{the:complementary_adjoint}
  Let \(\Tri\), \(\Addi\), and~\(F\) be as above, that is,
  \(\Tri\) is a triangulated category with countable direct
  sums, \(\Addi\) is either a stable Abelian category or a
  triangulated category, and \(F\colon \Tri\to\Addi\) is a
  stable functor commuting with countable direct sums and
  either homological \textup{(}if~\(\Addi\) is
  Abelian\textup{)} or exact \textup{(}if~\(\Addi\) is
  triangulated\textup{)}.  Let \(\Ideal_F\defeq \ker F\) and
  let~\(\Null_F\) be the class of \(F\)\nb-contractible objects
  as above.

  Let \(\Proj\Addi\subseteq\Addi\) be a subcategory with two
  properties: first, for any \(A\inOb\Tri\), there exists an
  epimorphism \(P\to F(A)\) with \(P\inOb\Proj\Addi\);
  secondly, the left adjoint functor~\(F^\lad\) of~\(F\) is
  defined on \(\Proj\Addi\), that is, for each
  \(P\in\Proj\Addi\), there is an object~\(F^\lad(P)\)
  in~\(\Tri\) with \(\Tri(F^\lad(P),B) \cong
  \Tri\bigl(P,F(B)\bigr)\) naturally for all \(B\inOb\Tri\).

  Then~\(\Ideal_F\) has enough projective objects, the
  subcategory~\(\Null_F\) is reflective, and the pair of
  localising subcategories
  \(\bigl(\gen{F^\lad(\Proj\Addi)},\Null_F\bigr)\) is
  complementary.
\end{theorem}

\begin{proof}
  \cite{Meyer-Nest:Homology_in_KK}*{Proposition 3.37} shows
  that~\(\Ideal_F\) has enough projective objects and that any
  projective object is a direct summand of \(F^\lad(P)\) for
  some \(P\in\Proj\Addi\).  Now
  Theorem~\ref{the:ideal_sum_complementary} yields the
  assertions.
\end{proof}

Theorem~\ref{the:complementary_adjoint} is non-trivial even
if~\(\Null_F\) contains only zero objects, that is, if
\(F(A)=0\) implies \(A=0\).  Then it asserts
\(\gen{F^\lad(\Proj\Addi)}=\Tri\).

In the situation of Theorem~\ref{the:complementary_adjoint}, we
also understand how objects of \(\gen{F^\lad(\Proj\Addi)}\)
are to be constructed from the building blocks in
\(F^\lad(\Proj\Addi)\).

Let \(\Cat_1\star\Cat_2\) for subcategories
\(\Cat_1,\Cat_2\subseteq\Tri\) be the subcategory of all
objects~\(A\) for which there is an exact sequence \(A_1\to
A\to A_2\to A_1[1]\) with \(A_1\inOb\Cat_1\) and
\(A_2\inOb\Cat_2\).  We abbreviate \(\Proj_F\defeq
F^\lad(\Proj\Addi)\) and recursively define \(\Proj_F^{\star
  n}\) for \(n\in\N\) by \(\Proj_F^{\star 0} \defeq \{0\}\) and
\(\Proj_F^{\star n} \defeq \Proj_F^{\star n-1}\star\Proj_F\)
for \(n\ge1\).

\begin{theorem}
  \label{the:adjoint_limit}
  In the situation of
  Theorem~\textup{\ref{the:complementary_adjoint}}, any object
  of \(\gen{F^\lad(\Proj\Addi)}\) is a homotopy colimit of an
  inductive system \((A_n)_{n\in\N}\) with
  \(A_n\in\Proj_F^{\star n}\).
\end{theorem}

\begin{proof}
  This follows from Theorem~\ref{the:cellular_limit_tower}.
  But an extra observation is needed here because we do not
  adjoin direct summands of objects in the definition of the
  subcategories~\(\Proj_F^{\star n}\), so that they do not
  necessarily contain all \(\Ideal_F^n\)\nb-projective objects.

  There is an \(\Ideal_F\)\nb-projective resolution with
  entries in~\(\Proj_F\), which we embed in a phantom castle.
  The resulting cellular approximation tower satisfies
  \(\tilde{A}_n\inOb\Proj_F^{\star n}\), so that
  Theorem~\ref{the:cellular_limit_tower} yields an inductive
  system of the required form.
\end{proof}

We have considered two cases above: homological and exact
functors.  For homological functors with values in the category
of Abelian groups, our results were obtained previously by
Apostolos Beligiannis~\cite{Beligiannis:Relative}.  Let
\(F\colon \Tri\to\Ab^\Z\) be a stable homological functor that
commutes with direct sums.  Suppose that~\(F\) is defined on
sufficiently many objects.  Then there must be a surjective map
\(A\prto \Z\) for which \(F^\lad(A)\) is defined.
Since~\(\Z\) is projective, \(\Z\) is a retract of~\(A\).
Since we assume~\(\Tri\) to have direct sums, idempotent
morphisms in~\(\Tri\) have range objects.
Thus~\(F^\lad(\Z)\) is defined as well.  By definition,
\(F^\lad(\Z)\) is a representing object for~\(F\).
Conversely, if~\(F\) is representable, then \(F^\lad\) can
be defined on all free Abelian groups.  Hence the
adjoint~\(F^\lad\) is defined on sufficiently many objects
if and only if~\(F\) is representable.  Furthermore, we can
take~\(\Proj_F\) to be the set of all direct sums of translated
copies of the representing object \(F^\lad(\Z)\).  The
assumption that~\(F\) commute with direct sums means that
\(F^\lad(\Z)\) is a compact object.

Summing up, if~\(F\) is a stable homological functor
to~\(\Ab^\Z\), then our methods apply if and only if
\(F(A)\cong \Tri_*(D,A)\) for a compact generator~\(D\)
of~\(\Tri\).  This situation is considered already
in~\cite{Beligiannis:Relative}.

\section{The ABC spectral sequence}
\label{sec:ABC_spectral_sequence}

When we apply a homological or cohomological functor to the
phantom tower, we get first an exact couple and then a spectral
sequence.  We call it the ABC spectral sequence after Adams,
Brinkmann, and Christensen.  Its second page only involves
derived functors.  The higher pages can be described in terms
of the phantom tower, but are more complicated.  It is
remarkable that the ABC spectral sequence is well-defined and
functorial on the level of triangulated categories, that is,
all the higher boundary maps are uniquely determined and
functorial without introducing finer structure like, say, model
categories.

Several results in this section are already know to the experts
or can be extracted from \cites{Beligiannis:Relative,
  Boardman:Conditionally, Christensen:Ideals}.  We have
included them, nevertheless, to give a reasonably
self-contained account.

\subsection{A spectral sequence from the phantom tower}
\label{sec:ABC_phantom}

We are going to construct exact couples out of the phantom
tower, extending results of Daniel
Christensen~\cite{Christensen:Ideals}.  We fix \(A\inOb\Tri\)
and a phantom tower~\eqref{eq:phantom_tower} over~\(A\).  In
addition, we let
\[
P_n\defeq 0,
\qquad
N_n\defeq A,
\quad\text{and}\quad
\iota_n^{n+1}\defeq \ID_A
\]
for \(n<0\).  Thus the triangles
\begin{equation}
  \label{eq:tri_PNN}
  P_n \xrightarrow{\pi_n}
  N_n \xrightarrow{\iota_n^{n+1}}
  N_{n+1} \xrightarrow{\epsilon_n}
  P_n[1]
\end{equation}
are exact for all \(n\in\Z\).  Of course, \(\iota_n^{n+1}\)
rarely belongs to~\(\Ideal\) for \(n<0\).

Let \(F\colon \Tri\to\Ab\) be a homological functor.  Define
bigraded Abelian groups
\begin{alignat*}{2}
  D &\defeq \sum_{p,q\in\Z} D_{pq},
  &\qquad
  D_{pq} &\defeq F_{p+q+1}(N_{p+1}),\\
  E &\defeq \sum_{p,q\in\Z} E_{pq},
  &\qquad
  E_{pq} &\defeq F_{p+q}(P_p),
\end{alignat*}
and homogeneous group homomorphisms
\[
\begin{gathered}
  \xymatrix@C=1em{
    D \ar[rr]^{i} & & D \ar[dl]^{j} \\ & E \ar[ul]^{k}
  }
\end{gathered}
\qquad
\begin{alignedat}{3}
  i_{pq}&\defeq (\iota_{p+1}^{p+2})_*\colon&
  D_{p,q} &\to D_{p+1,q-1},
  &\qquad \deg i &= (1,-1),
  \\
  j_{pq}&\defeq (\epsilon_p)_*\colon&
  D_{p,q} &\to E_{p,q},
  &\qquad \deg j &= (0,0),
  \\
  k_{pq}&\defeq (\pi_p)_*\colon&
  E_{p,q} &\to D_{p-1,q},
  &\qquad \deg k &= (-1,0).
\end{alignedat}
\]

Since~\(F\) is homological and the triangles~\eqref{eq:tri_PNN}
are exact, the chain complexes
\[\xymatrix{
  \cdots \ar[r] &
  F_m(P_n) \ar[r]^-{\pi_{n*}} &
  F_m(N_n) \ar[r]^-{\iota_{n*}^{n+1}} &
  F_m(N_{n+1}) \ar[r]^-{\epsilon_{n*}} &
  F_{m-1}(P_n) \ar[r] & \cdots
}
\]
are exact for all \(m\in\Z\).  Hence the data \((D,E,i,j,k)\)
above is an exact couple (see \cite{MacLane:Homology}*{Section
  XI.5}).

We briefly recall how an exact couple yields a spectral
sequence, see also \cite{MacLane:Homology}*{page 336--337}
or~\cite{Boardman:Conditionally}.  The first step is to form
\emph{derived exact couples}.  Let
\[
D^r\defeq i^{r-1}(D) \subseteq D,
\qquad
E^r\defeq k^{-1}(D^{r}) \bigm/ j(\ker i^{r-1}),
\]
for all \(r\ge1\).  Let \(i^{(r)}\colon D^r\to D^r\) be the
restriction of~\(i\); let \(k^{(r)}\colon E^r\to D^r\) be
induced by \(k\colon E\to D\); and let \(j^{(r)}\colon D^r\to
E^r\) be induced by the relation \(j\circ i^{1-r}\).  It is
shown in~\cite{MacLane:Homology} that \(E^{r+1}\cong
H(E^r,d^{(r)})\) for all \(r\in\N\), where
\(d^{(r)}=j^{(r)}k^{(r)}\); the map~\(d^{(r)}\) has bidegree
\((-r,r-1)\).  We call this spectral sequence the
\emph{ABC spectral sequence} for \(F\) and~\(A\).

We are going to describe the derived exact couples explicitly.
First, we claim that
\begin{equation}
  \label{eq:derived_homological}
  D_{p-1,q}^{r+1} \cong
  \begin{cases}
    \displaystyle F_{p+q}(A)
    &\text{for \(p\le0\),}\\
    F_{p+q}(A) \bigm/ F_{p+q}:\Ideal^p(A)
    &\text{for \(0\le p\le r\),}\\
    F_{p+q}(N_{p-r}) \bigm/ F_{p+q}:\Ideal^r(N_{p-r})
    &\text{for \(r\le p\).}
  \end{cases}
\end{equation}
By definition, \(D^{r+1}_{p-1,q}\) for \(r\in\N\) is the range
of the map \(i^r\colon D_{p-1-r,q+r}\to D_{p-1,q}\).  This is
the identity map on \(F_{p+q}(A)\) for \(p\le 0\), the map
\(\iota^p_*\colon F_{p+q}(A)\to F_{p+q}(N_p)\) for \(0\le p \le
r\), and the map \((\iota_{p-r}^p)_*\colon F_{p+q}(N_{p-r}) \to
F_{p+q}(N_p)\) for \(r\le p\).  Now recall that the maps
\(\iota_m^n\) are \(\Ideal^{n-m}\)-versal for all \(n\ge
m\ge0\) and use Lemma~\ref{lem:ideal_versal}.

\begin{proposition}
  \label{pro:ABC_spese}
  Let \(1\le r<\infty\).  Then we have \(E^{r+1}_{pq}\cong0\)
  for \(p\le -1\); for \(0\le p\le r\), there is an exact
  sequence
  \begin{multline*}
    0 \to
    \frac{F_{p+q+1}(N_{p})}{F_{p+q+1}:\Ideal^{r+1}(N_{p})}
    \xrightarrow{(\iota_p^{p+1})_*}
    \frac{F_{p+q+1}(N_{p+1})}{F_{p+q+1}:\Ideal^r(N_{p+1})}
    \\ \to E^{r+1}_{pq} \to
    \frac{F_{p+q}:\Ideal^{p+1}(A)}{F_{p+q}:\Ideal^p(A)}
    \to 0;
  \end{multline*}
  for \(p\ge r\), there is an exact sequence
  \begin{multline*}
    0 \to
    \frac{F_{p+q+1}(N_{p})}{F_{p+q+1}:\Ideal^{r+1}(N_{p})}
    \xrightarrow{(\iota_p^{p+1})_*}
    \frac{F_{p+q+1}(N_{p+1})}{F_{p+q+1}:\Ideal^r(N_{p+1})}
    \\ \to E^{r+1}_{pq} \to
    \frac{F_{p+q}:\Ideal^{r+1}(N_{p-r})}{F_{p+q}:\Ideal^r(N_{p-r})}
    \to 0.
  \end{multline*}
  Finally, for \(r=\infty\), let
  \[
  \Bad_{pq} \defeq
  F_q(N_p) \Bigm/ \bigcup_{r\in\N} F_q:\Ideal^r(N_p);
  \]
  then we get \(E^{\infty}_{pq}\cong0\) for \(p\le -1\), and
  exact sequences
  \[
  0 \to \Bad_{p,p+q+1}\to\Bad_{p+1,p+q+1}
  \to E^\infty_{pq}
  \to \frac{F_{p+q}:\Ideal^{p+1}(A)}{F_{p+q}:\Ideal^p(A)}
  \to 0.
  \]
  Therefore, if \(\bigcup_{r\in\N} F_q:\Ideal^r(N_p) =
  F_q(N_p)\) for all \(p\in\N\), then
  \[
  E^\infty_{pq} \cong
  \frac{F_{p+q}:\Ideal^{p+1}(A)}{F_{p+q}:\Ideal^p(A)},
  \]
  that is, the ABC spectral sequence converges towards
  \(F_m(A)\), and the induced increasing filtration on the
  limit is \(\bigl(F_m:\Ideal^r(A)\bigr)_{r\in\N}\).
\end{proposition}

\begin{proof}
  We have \(E^{r+1}_{pq}\cong0\) for \(p\le -1\) because
  already \(E_{pq}=E^1_{pq}=0\).  Let \(p\ge0\).  We
  use~\eqref{eq:derived_homological} and the exactness of the
  derived exact couple \((D^{r+1},E^{r+1})\) to
  compute~\(E^{r+1}\) by an extension involving the kernel and
  cokernel of the restriction of~\(i\) to \(D^{r+1}\).
  Since~\(\iota_m^{m+1}\) is \(\Ideal\)\nb-versal, \(x\in
  F(N_m)\) satisfies \((\iota_m^{m+1})_*(x)\in
  F:\Ideal^r(N_{m+1})\) if and only if \(x\in
  F:\Ideal^{r+1}(N_m)\).  Plugging this into the extension that
  describes~\(E^{r+1}\), we get the assertion, at least for
  finite~\(r\).

  The case \(r=\infty\) is similar.  Now
  \begin{equation}
    \label{eq:def_Einfty}
    E^\infty \defeq \bigcap_{r\in\N} k^{-1}(i^r D) \Bigm/
    \bigcup_{r\in\N} j(\ker i^r).
  \end{equation}
  The injectivity of the map
  \(\Bad_{p,p+q+1}\to\Bad_{p+1,p+q+1}\) follows from the
  exactness of colimits of Abelian groups.  Using \(i(D)=\ker
  j\), \(\ker i = k(E)\), and~\eqref{eq:def_Einfty}, we get a
  short exact sequence
  \begin{multline}
    \label{eq:Einfty_extension}
    0\to
    D_{pq} \Bigm/ \Bigl(i(D_{p-1,q+1}) + \bigcup \ker i^r\Bigr)
    \to E^\infty_{pq}
    \\ \to D_{p-1,q}\cap \ker i\cap \bigcap i^r(D) \to 0;
  \end{multline}
  the first map is induced by~\(j\), the second one by~\(k\).

  The intersection \(\bigcap i^r(D)\) is described
  by~\eqref{eq:derived_homological} for \(r=\infty\), so that
  the third case in~\eqref{eq:derived_homological} is missing.
  Hence the quotient in~\eqref{eq:Einfty_extension} is
  \[
  \ker i \cap \bigcap i^r(D)
  \cong F_{p+q}:\Ideal^{p+1}(A)\bigm/ F_{p+q}:\Ideal^p(A).
  \]
  The versality of the maps~\(\iota_m^n\) yields \(\bigcup \ker
  i^r \cap D_{p-1,q} = \bigcup_r F_{p+q}:\Ideal^r(N_p)\).
  Hence the kernel in~\eqref{eq:Einfty_extension} is
  \(\Bad_{p+1,p+q+1}/\Bad_{p,p+q+1}\).  If \(\Bad=0\), then the
  groups \(E^\infty_{pq}\) for \(p+q=m\) are the subquotients
  of the filtration \(F_m:\Ideal^p(A)\) on \(F_m(A)\).
  Moreover, since \(\Bad_{0m} = F_m(A) \Bigm/ \bigcup_{p\in\N}
  F_m:\Ideal^p(A)\), our hypothesis includes the statement that
  \(F_m(A) = \bigcup_{p\in\N} F_m:\Ideal^p(A)\).
\end{proof}

Dual constructions apply to a cohomological functor
\(\contra\colon \Tri^\op\to\Ab\).
Equation~\eqref{eq:tri_PNN} yields a sequence of exact chain
complexes
\[\xymatrix{
  \cdots \ar[r] &
  \contra^{m-1}(P_n) \ar[r]^-{\epsilon_n^*} &
  \contra^m(N_{n+1}) \ar[r]^-{(\iota_n^{n+1})^*} &
  \contra^m(N_n) \ar[r]^-{\pi_n^*} &
  \contra^m(P_n) \ar[r] & \cdots.
}
\]
Therefore, the following defines an exact couple:
\[
\tilde{D}^{pq} \defeq \contra^{p+q+1}(N_{p+1}),
\qquad
\tilde{E}^{pq} \defeq \contra^{p+q}(P_p),
\]
\[
\begin{gathered}
  \xymatrix@C=1em{
    \tilde{D} \ar[rr]^{i}&&
    \tilde{D} \ar[dl]^{j}\\&
    \tilde{E} \ar[ul]^{k}
  }
\end{gathered}
\qquad
\begin{alignedat}{3}
  i^{pq} &\defeq (\iota_p^{p+1})^*\colon&
  \tilde{D}^{p,q} &\to \tilde{D}^{p-1,q+1},
  &\qquad \deg i &= (-1,1),
  \\
  j^{pq} &\defeq \pi_{p+1}^*\colon&
  \tilde{D}^{p,q} &\to \tilde{E}^{p+1,q},
  &\qquad \deg j &= (1,0),
  \\
  k^{pq} &\defeq \epsilon_p^*\colon&
  \tilde{E}^{p,q} &\to \tilde{D}^{p,q},
  &\qquad \deg k &= (0,0).
\end{alignedat}
\]
Again we form derived exact couples
\((\tilde{D}_r,\tilde{E}_r,i_r,j_r,k_r)\), and
\((\tilde{E}_r,d_r)\) with \(d_r\defeq j_rk_r\) is a spectral
sequence.  The map~\(d_r\) has bidegree \((r,1-r)\).  We call
this spectral sequence the \emph{ABC spectral sequence} for
\(\contra\) and~\(A\).

We can describe the derived exact couples as above.  To begin
with,
\begin{equation}
  \label{eq:derived_cohomological}
  \tilde{D}_{r+1}^{p-1,q} =
  \begin{cases}
    \displaystyle \contra^{p+q}(A)
    &\text{for \(p\le0\),}\\
    \displaystyle \Ideal^p \contra^{p+q}(A)
    &\text{for \(0\le p\le r\),}\\
    \displaystyle \Ideal^r \contra^{p+q}(N_{p-r})
    &\text{for \(r\le p\).}
  \end{cases}
\end{equation}

\begin{proposition}
  \label{pro:ABC_spese_dual}
  Let \(1\le r<\infty\).  Then \(\tilde{E}_{r+1}^{pq}\cong0\)
  for \(p\le -1\), and there are exact sequences
  \[
  0 \to
  \frac{\Ideal^{p} \contra^{p+q}(A)}
  {\Ideal^{p+1} \contra^{p+q}(A)} \to
  \tilde{E}_{r+1}^{pq} \to
  \Ideal^{r} \contra^{p+q+1}(N_{p+1})
  \xrightarrow{(\iota_p^{p+1})^*}
  \Ideal^{r+1} \contra^{p+q+1}(N_{p}) \to 0
  \]
  for \(0\le p\le r\) and
  \[
  0 \to
  \frac{\Ideal^r \contra^{p+q}(N_{p-r})}
  {\Ideal^{r+1} \contra^{p+q}(N_{p-r})} \to
  \tilde{E}_{r+1}^{pq} \to
  \Ideal^{r} \contra^{p+q+1}(N_{p+1})
  \xrightarrow{(\iota_p^{p+1})^*}
  \Ideal^{r+1} \contra^{p+q+1}(N_{p}) \to 0
  \]
  for \(p\ge r\).  For \(r=\infty\), let
  \[
  \smash{\widetilde{\Bad}}^{pq} \defeq
  \bigcap_{r\in\N} \Ideal^r \contra^q(N_p),
  \]
  then we get \(E^{\infty}_{pq}\cong0\) for \(p\le -1\), and
  exact sequences
  \[
  0 \to \frac{\Ideal^p\contra_{p+q}(A)}
  {\Ideal^{p+1}\contra_{p+q}(A)}
  \to \tilde{E}_\infty^{pq}
  \\\to \smash{\widetilde{\Bad}}^{p+1,p+q+1}\to
  \smash{\widetilde{\Bad}}^{p,p+q+1}.
  \]
  Therefore, if \(\bigcap_{r\in\N} \Ideal^r
  \contra^q(N_p)=0\) for all \(p,q\), then
  \[
  \tilde{E}_\infty^{pq} \cong
  \frac{\Ideal^p\contra_{p+q}(A)}{\Ideal^{p+1}\contra_{p+q}(A)},
  \]
  that is, the ABC spectral sequence converges towards
  \(\contra^m(A)\), and the induced decreasing filtration on
  the limit is \(\bigl(\Ideal^r\contra^m(A)\bigr)_{r\in\N}\).
\end{proposition}

\begin{proof}
  This is proved exactly as in the homological case.
\end{proof}

Notice that we do not claim that the maps
\(\smash{\widetilde{\Bad}}^{p+1,p+q+1}\to
\smash{\widetilde{\Bad}}^{p,p+q+1}\) are surjective.  If~\(G\)
is representable, that is, \(G(A)\cong \Tri(A,B)\) for some
\(B\inOb\Tri\), then the question whether the maps
\(\smash{\widetilde{\Bad}}^{p+1,p+q+1}\to
\smash{\widetilde{\Bad}}^{p,p+q+1}\) are surjective is related
to the question whether
\(\Ideal^\infty\cdot\Ideal=\Ideal^\infty\).  In this case,
\(\smash{\widetilde{\Bad}}^{p,*} = \Ideal^\infty(N_p[*],B)\).
Since the maps~\(\iota_n^{n+1}\) are \(\Ideal\)\nb-versal,
\[
\range \bigl(\smash{\widetilde{\Bad}}^{p+1,*}\to
\smash{\widetilde{\Bad}}^{p,*}\bigr)
= \Ideal^\infty\circ\Ideal(N_p[*],B)
\subseteq \Ideal^\infty(N_p[*],B).
\]

\begin{theorem}
  \label{the:Etwo}
  Starting with the second page, the ABC spectral sequences
  for homological and cohomological functors are independent of
  auxiliary choices and functorial in~\(A\).  Their second
  pages contain the derived functors:
  \[
  E^2_{pq} \cong \Left_pF_q(A),
  \qquad
  \tilde{E}_2^{pq} \cong \Right^p\contra^q(A).
  \]
\end{theorem}

\begin{proof}
  We only formulate the proof for homological functors; the
  cohomological case is similar.  The map \(d\defeq jk\colon
  E\to E\) is induced by the maps \(\delta_p\colon P_{p+1}\to
  P_p[1]\) in the phantom tower.  By
  Lemma~\ref{lem:phantom_tower}, these maps form a
  \(\Proj\)\nb-projective resolution of~\(A\).  This together
  with counting of suspensions yields the description
  of~\(E^2_{pq}\).

  Let \(f\colon A\to A'\) be a morphism in~\(\Tri\).  By
  Lemma~\ref{lem:phantom_tower}, it lifts to a morphism between
  the phantom towers over \(A\) and~\(A'\).  This induces a
  morphism of exact couples and hence a morphism of spectral
  sequences.  The maps \(P_n\to P_n'\) form a chain map between
  the \(\Ideal\)\nb-projective resolutions embedded in the
  phantom towers.  This chain map lifting of~\(f\) is unique up
  to chain homotopy (see
  \cite{Meyer-Nest:Homology_in_KK}*{Proposition 3.36}).  Hence
  the induced map on~\(E^2\) is unique and functorial.  We
  get~\(E^r\) for \(r\in\N_{\ge2}\cup\{\infty\}\) as
  subquotients of~\(E^2\).  Since our map on~\(E^2\) is part of
  a morphism of exact couples, it descends to these
  subquotients in a unique and functorial way.  Thus~\(E^r\) is
  functorial for all \(r\ge2\).
\end{proof}

The naturality of the ABC spectral sequence does not mean that
the exact sequences in Propositions \ref{pro:ABC_spese}
and~\ref{pro:ABC_spese_dual} are natural.  They use the exact
couple underlying the ABC spectral sequence, and this exact
couple is not natural.  It is easy to check that the maps
\(E^{r+1}_{pq} \to F_{p+q}:\Ideal^{p+1}(A) \bigm/
F_{p+q}:\Ideal^p(A)\) in Proposition~\ref{pro:ABC_spese} are
canonical for \(0\le p\le r\le\infty\).  But
\(F_{p+q}:\Ideal^{r+1}(N_{p-r}) \bigm/
F_{p+q}:\Ideal^r(N_{p-r})\) depends on auxiliary choices.

Our results so far only \emph{formulate} the convergence
problem for the ABC spectral sequence.  It remains to check
whether the relevant obstructions vanish.  The easy special
case where the projective resolutions have finite length is
already dealt with in~\cite{Christensen:Ideals}.  Recall
that~\(\Proj_\Ideal^n\) denotes the class of
\(\Ideal^n\)\nb-projective objects.

\begin{lemma}
  \label{lem:tower_collapses}
  Let \(k\in\N\) and \(A\inOb\Proj_\Ideal^k\).  Then
  \(\iota_m^{m+k}=0\) and \(N_m\inOb\Proj_\Ideal^k\) for all
  \(m\in\N\).
\end{lemma}

\begin{proof}
  Since \(\iota_m^{m+k}\) is \(\Ideal^k\)\nb-versal, we have
  \(\iota_m^{m+k}=0\) if and only if
  \(N_m\inOb\Proj_\Ideal^k\).  We prove \(\iota_m^{m+k}=0\) by
  induction on~\(m\).  The case \(m=0\) is clear because
  \(N_0=A\).  If \(\iota_m^{m+k}=0\), then
  \(\iota_{m+1}^{m+k}\) factors through the map
  \(\epsilon_m\colon N_{m+1}\to P_m[1]\) by the long exact
  homology sequence for the triangles~\eqref{eq:tri_PNN}.  If
  we compose the resulting map \(P_m[1]\to N_{m+k}\) with
  \(\iota_{m+k}^{m+k+1}\in\Ideal\), we get zero because
  \(P_m[1]\inOb\Proj\).  Thus \(\iota_{m+1}^{m+k+1}\) vanishes
  as well.
\end{proof}

\begin{proposition}
  \label{pro:Ct_for_projective}
  Let \(F\colon \Tri\to\Ab\) be a homological functor and let
  \(m\in\N\).

  If \(A\inOb\Proj_\Ideal^{m+1}\), then the ABC spectral
  sequence for \(F\) and~\(A\) collapses at~\(E^{m+2}\) and
  converges towards \(F_*(A)\), and its limiting page
  \(E^\infty=E^{m+2}\) is supported in the region \(0\le p\le
  m\).

  If, instead, \(A\) has a \(\Proj\)\nb-projective resolution
  of length~\(m\), then the ABC spectral sequence for \(F\)
  and~\(A\) is supported in the region \(0\le p\le m\) from the
  second page onward, so that it collapses at~\(E^{m+1}\).  If,
  in addition, \(A\) belongs to the localising subcategory
  of~\(\Tri\) generated by~\(\Proj_\Ideal\), then the spectral
  sequence converges towards \(F_*(A)\).

  Similar assertions hold in the cohomological case.
\end{proposition}

\begin{proof}
  We only write down the argument in the homological case.  If
  \(A\inOb\Proj_\Ideal^{m+1}\), then
  \(N_p\inOb\Proj_\Ideal^{m+1}\) for all \(p\in\N\) by
  Lemma~\ref{lem:tower_collapses}.  Therefore,
  \(F:\Ideal^r(N_p)=F(N_p)\) for all \(r\ge m+1\).  Plugging
  this into Proposition~\ref{pro:ABC_spese}, we get
  \(E^r_{pq}=0\) for \(r\ge m+2\) and \(p\ge m+1\).  This
  forces the boundary maps~\(d^r\) to vanish for \(r\ge m+2\),
  so that \(E^\infty_{pq}=E^{m+2}_{pq}\).

  Suppose now that~\(A\) has a projective resolution of
  length~\(m\).  Embed such a resolution in a phantom tower by
  Lemma~\ref{lem:phantom_tower}, so that \(P_p=0\) and
  \(N_p\cong N_{p+1}\) for \(p>m\).  Then~\(E^r\) is supported
  in the region \(0\le p\le m\) for all \(r\ge1\).  For
  \(r\ge2\), this holds for any choice of phantom tower by
  Theorem~\ref{the:Etwo}.  As a consequence, \(d^r=0\) for
  \(r\ge m+1\) and hence \(E^\infty=E^{m+1}\).

  Suppose, in addition, that~\(A\) belongs to the localising
  subcategory of~\(\Tri\) generated by~\(\Proj_\Ideal\).  We
  claim that \(N_p\cong0\) for \(p>m\).  This implies that the
  ABC spectral sequence converges towards \(F_*(A)\).  If
  \(D\in\Proj_\Ideal\), then there are exact sequences
  \[
  \Tri_{*+1}(D,N_p) \into \Tri_*(D,P_{p-1}) \prto
  \Tri_*(D,N_{p-1})
  \]
  for all \(p\in\N_{\ge1}\) (see the proof of
  Lemma~\ref{lem:phantom_tower}).  Therefore,
  \(\Tri_*(D,N_p)=0\) for \(p>m\) if \(D\inOb\Proj_\Ideal\).
  The class of objects~\(D\) with this property is localising,
  that is, closed under suspensions, direct sums, direct
  summands, and exact triangles.  Hence it contains the
  localising subcategory generated by~\(\Proj_\Ideal\).  This
  includes \(A=N_0\) by assumption.  Since it contains
  all~\(P_n\) as well, it contains~\(N_n\) for all \(n\in\N\)
  because of the exact triangles in the phantom tower, so that
  \(\Tri_*(N_n,N_p)\) vanishes for all \(n\in\N\).  Thus
  \(\Tri_*(N_p,N_p)=0\), forcing \(N_p=0\).
\end{proof}

If~\(A\) belongs to the localising subcategory generated by the
\(\Ideal\)\nb-projective objects, then the existence of a
projective resolution of length~\(n\) implies that~\(A\) is
\(\Ideal^n\)\nb-projective.  This fails without an additional
hypothesis because \(\Ideal\)\nb-contractible objects have
projective resolutions of length~\(0\).

The converse assertion is usually far from true
(see~\cite{Christensen:Ideals}).
Proposition~\ref{pro:product_class} shows that an object~\(A\)
of~\(\Tri\) is \(\Ideal^2\)\nb-projective if and only if there
is an exact triangle \(P_2\to P_1\to A\to P_2[1]\) with
\(\Ideal\)\nb-projective objects \(P_1\) and~\(P_2\).  The
resulting chain complex \(0\to P_2\to P_1\to A\) is an
\(\Ideal\)\nb-projective resolution if and only if the map
\(A\to P_2[1]\) is an \(\Ideal\)\nb-phantom map, if and only if
the map \(P_2\to P_1\) is \(\Ideal\)\nb-monic.  But this need
not be the case in general.

Recall that the derived functors of the contravariant functor
\(A\mapsto \Tri_*(A,B)\) are the \emph{extension groups}
\(\Ext^*_{\Tri,\Ideal}(A,B)\).  These agree with extension
groups in the Abelian approximation, that is, the target
category of the universal \(\Ideal\)\nb-exact stable
homological functor.  In particular,
\(\Ext^0_{\Tri,\Ideal}(A,B)\) is the space of morphisms between
the images of \(A\) and~\(B\) in the Abelian approximation
(see~\cite{Meyer-Nest:Homology_in_KK}).  Theorem~\ref{the:Etwo}
and Proposition~\ref{pro:ABC_spese_dual} yield exact sequences
\begin{equation}
  \label{eq:Ext_0_exact}
  0 \to \frac{\Tri(A,B)}{\Ideal(A,B)}
  \to \Ext^0_{\Tri,\Ideal}(A,B)
  \to \Ideal(N_1[-1],B)
  \xrightarrow{(\iota_0^1)^*} \Ideal^2(A[-1],B) \to 0
\end{equation}
and
\begin{equation}
  \label{eq:Ext_1_exact}
  0 \to
  \frac{\Ideal(A,B)}{\Ideal^2(A,B)}
  \to \Ext^1_{\Tri,\Ideal}(A,B)
  \to \Ideal(N_2[-1],B)
  \xrightarrow{(\iota_1^2)^*} \Ideal^2(N_1[-1],B) \to 0.
\end{equation}
In particular, we get injective maps
\begin{align*}
  \Tri/\Ideal(A,B) &\to \Ext^0_{\Tri,\Ideal}(A,B),\\
  \Ideal/\Ideal^2(A,B) &\to \Ext^1_{\Tri,\Ideal}(A,B).
\end{align*}
But these maps need not be surjective.  What they do is easy to
understand.  The first map is simply the functor from~\(\Tri\)
to the Abelian approximation.  The second map embeds a morphism
in~\(\Ideal\) in an exact triangle.  This triangle is
\(\Ideal\)\nb-exact because it involves a phantom map, and
hence provides an extension in the Abelian approximation.

The higher quotients \(\Ideal^n/\Ideal^{n+1}(A,B)\) are also
related to \(\Ext^n_{\Tri,\Ideal}(A,B)\), but this is merely a
relation in the formal sense, that is, it is no longer a map.
To construct this relation, we use the
\(\Ideal^n\)\nb-versality of the map \(\iota^n\colon A\to
N_n\).  Thus any map \(f\in\Ideal^n(A,B)\) factors through a
map \(\hat{f}\colon N_n\to B\), and we can choose
\(\hat{f}\in\Ideal(N_n,B)\) if \(f\in\Ideal^{n+1}(A,B)\).  Now
we use the map
\[
\Tri/\Ideal(N_n,B) \to \Ext^n_{\Tri,\Ideal}(A,B)
\]
provided by Theorem~\ref{the:Etwo} and
Proposition~\ref{pro:ABC_spese_dual}.  Since~\(f\) does not
determine the class of~\(\hat{f}\) in \(\Tri/\Ideal(N_n,B)\)
uniquely, we only get a relation.  The ambiguity in this
relation disappears on the \(n\)th page of the ABC spectral
sequence by Proposition~\ref{pro:ABC_spese_dual}.

Once we have chosen~\(\hat{f}\) as above, we can also extend it
to a morphism between the phantom towers of \(A\) and~\(B\)
that shifts degrees down by~\(n\) -- the extension to the left
of~\(N_n\) is induced by \(\hat{f}\circ\iota_m^n\colon N_m\to
B\) for \(m<n\) and vanishes on~\(P_m\) for \(m<n\).  Thus we
get a morphism between the ABC spectral sequences for \(A\)
and~\(B\) for any homological or cohomological
functor -- shifting degrees down by~\(n\), of course.

\subsection{An equivalent exact couple}
\label{sec:phantom_castle_couple}

The cellular approximation tower produces a spectral sequence
in the same way as the phantom tower.

We extend the phantom tower to \(n<0\) by \(\tilde{A}_n=0\) and
\(P_n=0\) for all \(n<0\).  The triangles~\eqref{eq:tri_AAP}
are exact for all \(n\in\Z\).  A homological functor~\(F\) maps
these exact triangles to exact chain complexes
\[
\xymatrix{
  \cdots \ar[r] &
  F_m(\tilde{A}_n) \ar[r]^-{\alpha_{n*}^{n+1}} &
  F_m(\tilde{A}_{n+1}) \ar[r]^-{\sigma_{n*}} &
  F_m(P_n) \ar[r]^-{\kappa_{n*}} &
  F_{m-1}(\tilde{A}_n) \ar[r] &
  \cdots.
}
\]
As above, these amount to an exact couple
\[
\begin{gathered}
  \xymatrix@C=1em{
    D' \ar[rr]^{i'}&& D' \ar[dl]^{j'}\\& E' \ar[ul]^{k'}
  }
\end{gathered}
\qquad
\begin{aligned}
  D'_{pq} &\defeq F_{p+q}(\tilde{A}_{p+1}),\\
  E'_{pq} &\defeq F_{p+q}(P_p),
\end{aligned}
\qquad
\begin{alignedat}{2}
  i'_{pq}&\defeq
  (\alpha_{p+1}^{p+2})_*\colon& D'_{p,q} &\to D'_{p+1,q-1},\\
  j'_{pq}&\defeq (\sigma_p)_*\colon& D'_{p,q} &\to E'_{p,q},\\
  k'_{pq}&\defeq (\kappa_p)_*\colon& E'_{p,q} &\to D'_{p-1,q}.
\end{alignedat}
\]

Part of the commuting diagram~\eqref{eq:octahedral_relations}
asserts that the identity maps \(E\to E'\) and the maps
\(\gamma_{p+1,*}\colon D_{pq}\to D'_{pq}\) form a morphism of
exact couples between the exact couples from the phantom tower
and the cellular approximation tower.  This induces a morphism
between the resulting spectral sequences.  Since this morphism
acts identically on~\(E^1\), the induced morphisms on~\(E^r\)
must be invertible for all \(r\in\N_{\ge1}\cup\{\infty\}\).
Hence our new spectral sequence is isomorphic to the ABC
spectral sequence.

Although the spectral sequences are isomorphic, the underlying
exact couples are different and thus provide isomorphic but
different descriptions of~\(E^\infty\).

An important difference between the two exact couples is that
\(D'_{pq}=0\) for \(p\le0\).  Hence any element of \(D'_{pq}\)
is annihilated by a sufficiently high power of~\(i\).
Therefore, the kernel in~\eqref{eq:Einfty_extension} vanishes
and
\[
E^\infty_{pq}
\cong D'_{p-1,q} \cap \ker i' \cap \bigcap_{r\in\N} (i')^r(D').
\]
Let
\[
L_{pq} \defeq
\varinjlim \range\bigl(\alpha_{p*}^r\colon
F_q(\tilde{A}_p) \to F_q(\tilde{A}_r)\bigr);
\]
these spaces define an increasing filtration
\((F_{pq})_{p\in\N}\) on \(\varinjlim F_q(\tilde{A}_r)\) -- we
form the limit with the maps \(\alpha_m^n\).
Using~\eqref{eq:Einfty_extension} and the exactness of colimits
of Abelian groups, we get isomorphisms
\[
E^\infty_{pq}\cong \frac{L_{p+1,p+q}}{L_{p,p+q}}.
\]
Hence \(E^\infty_{pq}\) converges towards \(\varinjlim
F_q(\tilde{A}_r)\) and induces the filtration \((L_{p,p+q})\)
on its limit -- without any assumption on the ideal or the
homological functor.

In the cohomological case, the exact
triangles~\eqref{eq:tri_AAP} yield an exact couple as well, and
the morphism between the two exact couples from the cellular
approximation tower and the phantom tower induces an
isomorphism between the associated spectral sequences.  Again,
we get a new description of~\(\tilde{E}_\infty\).

But the result is not as simple as in the homological case
because the projective limit functor for Abelian groups is not
exact.  Let
\[
\tilde{L}^{pq} \defeq
\bigcap_{r\ge p} \range\bigl(\alpha_p^{r*}\colon
\contra^q(\tilde{A}_r) \to \contra^q(\tilde{A}_p)\bigr).
\]
Then
\[
\tilde{E}_\infty^{pq} \cong
\ker\bigl( \tilde{L}^{p+1,p+q} \to \tilde{L}^{p,p+q}\bigr).
\]
In general, we cannot say much more than this.  If
\(\tilde{L}^{pq}\) is the range of the map \(\varprojlim_r
\contra^q(\tilde{A}_r)\to \contra^q(\tilde{A}_p)\) for all
\(p\) and~\(q\), then the ABC spectral sequence converges
towards \(\varprojlim_r \contra^q(\tilde{A}_r)\) and induces on
this limit the decreasing filtration by the subspaces
\(\ker\bigl(\varprojlim_r \contra^q(\tilde{A}_r) \to
\contra^q(\tilde{A}_p)\bigr)\) for \(p\in\N\).

\section{Convergence of the ABC spectral sequence}
\label{sec:convergence_ABC}

There is an obvious obstruction to the convergence of the ABC
spectral sequence: the subcategory~\(\Null_\Ideal\) of
\(\Ideal\)\nb-contractible objects.  Since
\(\Ideal\)\nb-derived functors vanish on~\(\Null_\Ideal\), the
spectral sequence cannot converge towards the original functor
unless it vanishes on~\(\Null_\Ideal\) as well.  At best, the
ABC spectral sequence may converge to the \emph{localisation}
of the given functor at~\(\Null_\Ideal\).  We show that this is
indeed the case for homological functors that commute with
direct sums, provided the ideal~\(\Ideal\) is compatible with
direct sums.  The situation for cohomological functors is less
satisfactory because the projective limit functor for Abelian
groups is not exact.

We continue to assume throughout this section that the
category~\(\Tri\) has countable direct sums.  Various notions
of convergence of spectral sequences are discussed
in~\cite{Boardman:Conditionally}.  The following results deal
only with strong convergence.

\begin{theorem}
  \label{the:homological_limit}
  Let~\(\Ideal\) be a homological ideal compatible with direct
  sums in a triangulated category~\(\Tri\); let \(F\colon
  \Tri\to\Ab\) be a homological functor that commutes with
  countable direct sums, and let \(A\inOb\Tri\); let \(\Left
  F\) be the localisation of~\(F\) at~\(\Null_\Ideal\).  Then
  \[
  F:\Ideal^{2\infty}(A)
  = F:\Ideal^\infty(A)
  = \bigcup_{r\in\N} F:\Ideal^r(A)
  = \range\bigl(\Left F(A)\to F(A)\bigr),
  \]
  and the ABC spectral sequence for \(F\) and~\(A\) converges
  towards \(\Left F(A)\) with the filtration \(\bigl(\Left
  F:\Ideal^k(A)\bigr)_{k\in\N}\).  We have \(\bigcup_{k\in\N}
  \Left F:\Ideal^k(A) = \Left F(A)\).
\end{theorem}

\begin{proof}
  Lemma~\ref{lem:ideal_versal} implies that \(F:\Ideal^r(A)\)
  is the range of \(\alpha_{r*}\colon F(\tilde{A}_r)\to F(A)\)
  for all \(r\in\N\) and that \(F:\Ideal^\infty(A)\) is the
  range of \(F\bigl(\bigoplus \tilde{A}_r\bigr)\to F(A)\)
  because~\eqref{eq:Ideal_infty_resolution} is an
  \(\Ideal^\infty\)\nb-projective resolution.  Now
  \(F\bigl(\bigoplus \tilde{A}_r\bigr) = \bigoplus
  F(\tilde{A}_r)\) shows that \(F:\Ideal^\infty(A)\) is the
  union of \(F:\Ideal^r(A)\) for \(r\in\N\).

  Let~\(\tilde{A}\) be as in~\eqref{eq:holim_induce}, so that
  \(\tilde{A}=L(A)\) and \(\Left F(A) = F(\tilde{A})\).  Since
  the inductive limit functor for Abelian groups is exact, the
  map \(\ID-S\) on \(\bigoplus F(\tilde{A}_r)\) is injective
  and has cokernel \(\varinjlim F(\tilde{A}_r)\).  Since the
  top row in~\eqref{eq:holim_induce} is an exact triangle, the
  long exact sequence yields \(F(\tilde{A})\cong\varinjlim
  F(\tilde{A}_r)\).  As a consequence, the range of \(f_*\colon
  F(\tilde{A})\to F(A)\) is equal to \(F:\Ideal^\infty(A)\).
  Since~\(\tilde{A}\) is \(\Ideal^{2\infty}\)\nb-projective by
  Theorem~\ref{the:cellular_limit_tower} and \(f\colon
  \tilde{A}\to A\) is an \(\Ideal\)\nb-equivalence, this map is
  an \(\Ideal^{2\infty}\)\nb-projective resolution of~\(A\).
  Hence the range of~\(f_*\) also agrees with
  \(F:\Ideal^{2\infty}(A)\) by Lemma~\ref{lem:ideal_versal}.

  Especially, \(F:\Ideal^\infty(A)=F(A)\) if
  \(A\inOb\gen{\Proj_\Ideal}\).  For such~\(A\), all objects
  that occur in the phantom castle belong to
  \(\gen{\Proj_\Ideal}\) as well, so that \(F(N_p) =
  \bigcup_{r\in\N} F:\Ideal^r(N_p)\) for all \(p\in\N\).  Hence
  Proposition~\ref{pro:ABC_spese} yields the convergence of the
  ABC spectral sequence to \(F(A)\) as asserted.  Since the ABC
  spectral sequences for \(A\) and~\(\tilde{A}\) are isomorphic
  by Proposition~\ref{pro:localise_phantom_castle}, we get
  convergence towards \(\Left F(A)\) for general~\(A\).
\end{proof}

The convergence of the ABC spectral sequences is more
problematic for a cohomological functor \(G\colon
\Tri^\op\to\Ab\) because projective limits of Abelian groups
are not exact.  In the following, we \emph{assume
  that~\(\contra\) maps direct sums to direct products} -- this
is the correct compatibility with direct sums for contravariant
functors.

The exactness of the first row in~\eqref{eq:holim_induce}
yields an exact sequence
\begin{equation}
  \label{eq:Milnor_cohomological_ABC}
  \varprojlim\nolimits^1 \contra^{*-1}(\tilde{A}_n)
  \into \contra^*(\tilde{A})
  \prto \varprojlim\nolimits \contra^*(\tilde{A}_n)
\end{equation}
for any~\(A\) (this also follows from
\cite{Meyer-Nest:Homology_in_KK}*{Theorem 4.4} applied to the
ideal~\(\Ideal^\infty\)).  Furthermore, \(G^*(\tilde{A}) =
\Right G^*(A)\).  Since~\eqref{eq:holim_induce} is an
\(\Ideal^\infty\)\nb-projective resolution, we have
\[
\varprojlim\nolimits^1 \contra^{-1}(\tilde{A}_n)
\cong \Ideal^\infty\contra(\tilde{A}).
\]
The same argument as in the homological case yields
\begin{equation}
  \label{eq:contra_infty_intersect}
  \Ideal^\infty\contra(A) = \bigcap_{n\in\N} \Ideal^n\contra(A).
\end{equation}
Using compatibility of~\(\contra\) with direct sums, we can
also rewrite the obstructions to the convergence of the ABC
spectral sequence in Proposition~\ref{pro:ABC_spese_dual}:
\[
\smash{\widetilde{\Bad}}^{pq} \cong
\Ideal^\infty\contra^q(\tilde{N}_p),
\]
where~\(\tilde{N}_p\) is the \(p\)th object in a phantom tower
over~\(\tilde{A}\) instead of~\(A\).  The spectral sequence
converges towards \(\Right\contra(A)\) if these obstructions
all vanish.

\begin{proposition}
  \label{pro:ABC_dual_converges}
  Let~\(\Ideal\) be a homological ideal with enough projectives
  that is compatible with direct sums, and let \(\contra\colon
  \Tri^\op\to\Ab\) be a cohomological functor that maps direct
  sums to direct products.  Let \(A\inOb\Tri\) and let
  \(L(A)\inOb\gen{\Proj_\Ideal}\) be its
  \(\Proj_\Ideal\)\nb-cellular approximation.  If \(L(A)\) is
  \(\Ideal^\infty\)\nb-projective, then the ABC spectral
  sequence for \(A\) and~\(\contra\) converges towards
  \(\Right\contra(A) = \contra\circ L(A)\).
\end{proposition}

\begin{proof}
  Proposition~\ref{pro:localise_phantom_castle} implies that
  \(A\) and \(L(A)\) have canonically isomorphic ABC spectral
  sequences.  Hence we may replace~\(A\) by \(L(A)\) and assume
  that~\(A\) itself is \(\Ideal^\infty\)\nb-projective.  By
  Proposition~\ref{pro:intersect_class}, \(A\) is a direct
  summand of \(\bigoplus_{n\in\N} A_n\) with
  \(\Ideal^n\)\nb-projective objects~\(A_n\).  The ABC spectral
  sequence for each~\(A_n\) converges by
  Proposition~\ref{pro:Ct_for_projective}.

  Since~\(\Ideal\) is compatible with countable direct sums, a
  direct sum of phantom castles over~\(A_n\) is a phantom
  castle over \(\bigoplus A_n\).  Thus the ABC spectral
  sequence for \(\bigoplus_{n\in\N} A_n\) is the direct product
  of the ABC spectral sequences for~\(A_n\); here we use
  that~\(\contra\) maps direct sums to direct products.  Hence
  the ABC spectral sequence for \(\bigoplus_{n\in\N} A_n\)
  converges towards \(\prod \contra(A_n) =
  \contra\bigl(\bigoplus A_n\bigr)\).  Since the ABC spectral
  sequence is an additive functor on~\(\Tri\), this implies
  that the ABC spectral sequence for any direct summand of
  \(\bigoplus A_n\) converges.  This yields the convergence of
  the ABC spectral sequence for \(L(A)\), as desired.
\end{proof}

\section{A classical special case}
\label{sec:classical_application}

Before we apply our results to equivariant bivariant
\(\K\)\nb-theory, we briefly discuss a more classical
application in homological algebra, where we recover results by
Marcel B\"okstedt and Amnon Neeman~\cite{Boekstedt-Neeman} and
where the ABC sectral sequence specialises to a spectral
sequence due to Alexander Grothendieck.

Let~\(\Abel\) be an Abelian category with enough projective
objects and exact countable direct sums.  Let \(\Ho(\Abel)\) be
the homotopy category of chain complexes over~\(\Abel\).  We
require no finiteness conditions, so that \(\Ho(\Abel)\) is a
triangulated category with countable direct sums.  We are
interested in the derived category of~\(\Abel\) and therefore
want to localise at the full subcategory
\(\Null\subseteq\Ho(\Abel)\) of exact chain complexes.  This
subcategory is localising because countable direct sums of
exact chain complexes are again exact by assumption.

The obvious functor defining this subcategory~\(\Null\) is the
functor
\[
H\colon \Ho(\Abel)\to\Abel^\Z
\]
that maps a chain complex to its homology.  The functor~\(H\)
is a stable homological functor that commutes with direct sums.
Hence its kernel~\(\Ideal_H\) is a homological ideal that is
compatible with direct sums.

Let \(\Proj\Abel^\Z\subseteq\Abel^\Z\) be the full subcategory
of projective objects.  Since we assume~\(\Abel\) to have
enough projective objects, any object of \(\Abel^\Z\) admits an
epimorphism from an object in \(\Proj\Abel^\Z\).  It is easy to
see that the left adjoint of the homology functor is defined on
\(\Proj\Abel^\Z\) and maps a sequence \((P_n)\) of projective
objects to the chain complex \((P_n)\) with vanishing boundary
map.  Since this functor is clearly fully faithful, we use it
to view \(\Proj\Abel^\Z\) as a full subcategory of
\(\Ho(\Abel)\), omitting the functor~\(H^\lad\) from our
notation.  Using the criterion
of~\cite{Meyer-Nest:Homology_in_KK}, it is easy to check that
the functor~\(H\) above is the universal \(\Ideal_H\)\nb-exact
homological functor.

Theorems \ref{the:complementary_adjoint}
and~\ref{the:adjoint_limit} apply here.  The first one shows
that \((\gen{\Proj\Abel^\Z},\Null)\) is a complementary pair of
subcategories.  Thus \(\gen{\Proj\Abel^\Z}\) is equivalent to
the derived category of~\(\Abel\).  Furthermore, any object of
the derived category is a homotopy colimit of a diagram with
entries in \((\Proj\Abel^\Z)^{\star n}\) for \(n\in\N\).

Let \(F\colon \Abel\to\Ab\) be an additive covariant functor
that commutes with direct sums.  We extend~\(F\) to an exact
functor \(\Ho(F)\colon \Ho(\Abel)\to\Ho(\Ab)\).  Let
\[
\bar{F}_q=H_q\circ \Ho(F)\colon \Ho(\Abel)\to \Ab
\]
be the functor that maps a chain complex~\(C_\bullet\) to the
\(q\)th homology of \(\Ho(F)(C_\bullet)\).  This is a
homological functor.  Its derived functors with respect
to~\(\Ideal_H\) are computed
in~\cite{Meyer-Nest:Homology_in_KK}: for a chain
complex~\(C_\bullet\), we have
\[
\Left_p\bar{F}_q(C_\bullet) =
(\Left_pF)\bigl(H_q(C_\bullet)\bigr),
\]
that is, we apply the usual derived functors of~\(F\) to the
homology of~\(C_\bullet\).  Thus the ABC spectral sequence
computes the homology of the total derived functor of~\(F\)
applied to~\(C_\bullet\) in terms of the derived functors
of~\(F\), applied to \(H_*(C_\bullet)\). Such a spectral
sequence was already constructed by Alexander Grothendieck.

\section{Construction of the Baum--Connes assembly map}
\label{sec:applications_BC}

Finally, we apply our general machinery to construct the
Baum--Connes assembly map with coefficients first for locally
compact groups and then for certain discrete quantum groups.
In the group case, we get a simpler argument than
in~\cite{Meyer-Nest:BC}.

\subsection{The assembly map for locally compact groups}
\label{sec:assembly_group}

Let~\(G\) be a second countable locally compact group and let
\(\KKcat^G\) be the \(G\)\nb-\alb{}equivariant Kasparov
category; its objects are the separable \(\Cst\)\nb-algebras
with a strongly continuous action of~\(G\), its morphism space
\(A\to B\) is \(\KK^G_0(A,B)\).  It is shown
in~\cite{Meyer-Nest:BC} that this category is triangulated (we
must exclude \(\Z/2\)\nb-graded \(\Cst\)\nb-algebras for this).

The category \(\KKcat^G\) has countable direct sums -- they are
just direct sums of \(\Cst\)\nb-algebras.  But uncountable
direct sums usually do not exist because of the separability
assumption in the definition of \(\KKcat^G\), which is needed
to make the analysis work.  Alternative definitions of
bivariant \(\K\)\nb-theory by Joachim
Cuntz~\cite{Cuntz-Meyer-Rosenberg} still work for non-separable
\(\Cst\)\nb-algebras, but it is not clear whether direct sums of
\(\Cst\)\nb-algebras remain direct sums in this category because
the definition of the Kasparov groups for inseparable
\(\Cst\)\nb-algebras involves colimits, which do not commute
with the direct products that appear in the definition of the
direct sum.

With enough effort, it should be possible to extend
\(\KKcat^G\) to a category with uncountable direct sums.  But
it seems easier to avoid this by imposing cardinality
restrictions on direct sums.

For any closed subgroup \(H\subseteq G\), we have induction and
restriction functors
\[
\Ind_H^G\colon \KKcat^H\to\KKcat^G,
\qquad
\Res_G^H\colon \KKcat^G\to\KKcat^H;
\]
the latter functor is quite trivial and simply forgets part of
the group action.  These functors give rise to two
subcategories of \(\KKcat^G\), which play a crucial role
in~\cite{Meyer-Nest:BC}.

\begin{definition}
  \label{def:CC}
  Let~\(\Fam\) be the set of all compact subgroups of~\(G\).
  \begin{align*}
    \CC &\defeq \{A\inOb\KKcat^G \mid
    \text{\(\Res_G^H(A)=0\) for all \(H\in\Fam\)}\},\\
    \CI &\defeq \{\Ind_H^G(A) \mid \text{\(A\inOb\KKcat^H\) and
      \(H\in\Fam\)}\}.
  \end{align*}
\end{definition}

Whereas the subcategory~\(\CC\) is localising by definition,
\(\CI\) is not.  Therefore, the localising subcategory it
generates, \(\gen{\CI}\), plays an important role as well.
Since~\(G\) acts properly on objects of~\(\CI\), they satisfy
the Baum--Connes conjecture, that is, the Baum--Connes assembly
map is an isomorphism for coefficients in~\(\CI\).  Since
domain and target of the assembly map are exact functors on
\(\KKcat^G\), this extends to the category \(\gen{\CI}\).  On
objects of \(\CC\) the domain of the Baum--Connes assembly map
is known to vanish, so that the Baum--Connes conjecture
predicts \(\K_*(G\rcross A)=0\) for \(A\inOb\CC\).

On a technical level, the main tool in~\cite{Meyer-Nest:BC} is
that the pair of subcategories \((\gen{\CI},\CC)\) is
complementary.  Hence the Baum--Connes assembly map is
determined by what it does on these two subcategories.  This
implies that its domain is the localisation of the functor
\(A\mapsto \K_*(G\rcross A)\) at~\(\CC\) and that the assembly
map is the natural transformation from this localisation to the
original functor.

Put differently, the Baum--Connes assembly map is the only
natural transformation from an exact functor on \(\KKcat^G\) to
the functor \(\K_*(G\rcross\blank)\) that is an isomorphism on
\(\CI\) and whose domain vanishes on~\(\CC\) (we give some more
details about this argument in the related quantum group case
below).

In order to prove that \((\gen{\CI},\CC)\) is complementary, we
introduce the following ideal:

\begin{definition}
  \label{def:ideal_BC}
  Let \(\VC = \bigcap_{H\in\Fam} \ker \Res_G^H\).
\end{definition}

This ideal consists of the morphisms that \emph{vanish for
  compact subgroups} in the notation of~\cite{Meyer-Nest:BC}.
Clearly, an object belongs to~\(\CC\) if and only if its
identity map belongs to~\(\VC\), that is, \(\Null_\VC=\CC\).
Moreover, \cite{Meyer-Nest:BC}*{Proposition 4.4} implies that
objects of \(\CI\) are \(\VC\)\nb-projective; even more,
\(f\in\VC(A,B)\) if and only if~\(f\) induces the zero map
\(\KK^G_*(D,A)\to\KK^G_*(D,B)\) for all \(D\inOb\CI\).

We can also describe~\(\VC\) as the kernel of a single functor:
\[
F= (\Res_G^H)_{H\in\Fam}\colon
\KKcat^G\to \prod_{H\in\Fam} \KKcat^H.
\]
The functor~\(F\) commutes with direct sums because each
functor \(\Res_G^H\) clearly does so.  Hence~\(\VC\) is
compatible with countable direct sums.

The following theorem contains the main assertion in
\cite{Meyer-Nest:BC}*{Theorem 4.7}.  We will provide a simpler
proof here than in~\cite{Meyer-Nest:BC}.

\begin{theorem}
  \label{the:complementary_group_BC}
  The projective objects for~\(\VC\) are the retracts of direct
  sums of objects in \(\CI\), and the ideal~\(\VC\) has enough
  projective objects.  Hence the pair of subcategories
  \((\gen{\CI},\CC)\) is complementary.
\end{theorem}

\begin{proof}
  As in Theorem~\ref{the:complementary_adjoint}, we study the
  partially defined left adjoint of the functor~\(F\) above or,
  equivalently, of the functors \(\Res_G^H\) for \(H\in\Fam\).

  The discrete case is particularly simple because then all
  \(H\in\Fam\) are open subgroups.  If \(H\subseteq G\) is
  open, then \(\Ind_H^G\) is left adjoint to \(\Res_G^H\).
  Thus we may take \(\Proj\Addi= \prod_{H\in\Fam} \KKcat^H\) in
  Theorem~\ref{the:complementary_adjoint} and get
  \(F^\lad\bigl((A_H)_{H\in\Fam}\bigr) = \bigoplus_{H\in\Fam}
  \Ind_H^G(A_H)\).  Notice that the set~\(\Fam\) is countable
  if~\(G\) is discrete, so that this definition is legitimate.
  It follows that~\(\VC\) has enough projective objects and
  that they are all direct summands of \(\bigoplus_{H\in\Fam}
  \Ind_H^G(A_H)\) for suitable families~\((A_H)\), as asserted.

  For locally compact~\(G\), the argument gets more complicated
  because the functor \(\Res_G^H\) does not always have a left
  adjoint, and if it has, it need not be simply \(\Ind_H^G\).
  But there are still enough compact subgroups~\(H\) for which
  the left adjoint is defined on enough
  \(H\)\nb-\(\Cst\)-algebras and close enough to the induction
  functor for the argument above to go through.

  A good way to understand this is the duality theory developed
  in \cites{Emerson-Meyer:Euler, Emerson-Meyer:Dualities}.
  This is relevant because the induction functor provides an
  equivalence of categories \(\KKcat^H \simeq \KKcat^{G\ltimes
    G/H}\), where we use the groupoid \(G\ltimes G/H\), that
  is, we consider \(G\)\nb-equivariant bundles of
  \(\Cst\)\nb-algebras over~\(G/H\).  This equivalence of
  categories reflects the equivalence between the groupoids
  \(H\) and \(G\ltimes G/H\).

  Identifying \(\KKcat^H\simeq \KKcat^{G\ltimes G/H}\), the
  restriction functor \(\Res_G^H\) becomes the functor
  \(p_{G/H}^*\colon \KKcat^G\to \KKcat^{G\ltimes G/H}\) that
  pulls back a \(G\)\nb-\(\Cst\)-algebra to a trivial bundle of
  \(G\)\nb-\(\Cst\)-algebras over \(G/H\).
  Following~\cite{Kasparov:Novikov}, it is shown
  in~\cite{Emerson-Meyer:Euler} that the left adjoint of
  \(p^*_{G/H}\) is defined on all trivial bundles if~\(G/H\) is
  a smooth manifold.  We will see that this is enough for our
  purposes.

  As in~\cite{Meyer-Nest:BC}, we call a compact subgroup
  \emph{large} if it is a maximal compact subgroup in an open,
  almost connected subgroup of~\(G\).

  Let~\(H\) be large.  Then \(G/H\) is a smooth manifold and
  any compact subgroup is contained in a large one by
  \cite{Meyer-Nest:BC}*{Lemma 3.1}.  Furthermore, since~\(G\)
  is second countable there is a sequence \((U_n)_{n\in\N}\) of
  almost connected open subgroups of~\(G\) such that any other
  one is contained in~\(U_n\) for some \(n\in\N\).  Pick a
  maximal compact subgroup \(H_n\subseteq U_n\) for each
  \(n\in\N\).  Then any compact subgroup of~\(G\) is
  subconjugate to~\(H_n\) for some \(n\in\N\).  Therefore, we
  already have \(\VC=\bigcap_{n\in\N} \Res_G^{H_n}\) because
  \(\Res_G^K\) factors through \(\Res_G^H\) if~\(K\) is
  subconjugate to~\(H\).

  For a compact subgroup \(H\subseteq G\), let
  \(\RKKcat^G(G/H)\subseteq \KKcat^{G\ltimes G/H}\) be the full
  subcategory of trivial bundles over~\(G/H\) or, equivalently,
  the essential range of the functor \(p_{G/H}^*\).  We do not
  care whether this category is triangulated, it is certainly
  additive.  We replace the functors \(\Res_G^H\) by
  \(p_{G/H}^*\colon \KKcat^G\to\RKKcat^G(G/H)\) for
  \(H\in\Fam\).  For the large compact subgroups~\(H_n\)
  selected above, the results in
  \cites{Emerson-Meyer:Dualities} show that the left adjoint of
  \(p_{G/H}^*\) is defined on all of \(\RKKcat^G(G/H)\) and
  maps the trivial bundle with fibre~\(A\) to
  \(\CONT_0(T\,G/H)\otimes A\) with the diagonal action
  of~\(G\); here \(T\,G/H\) denotes the tangent space of
  \(G/H\) (we are not allowed to use the Clifford algebra dual
  considered in \cite{Emerson-Meyer:Euler} because it involves
  \(\Z/2\)\nb-graded \(\Cst\)\nb-algebras, which do not belong
  to our category).

  Thus we have verified the hypotheses of
  Theorem~\ref{the:complementary_adjoint} and can conclude
  that~\(\VC\) has enough projective objects and that~\(\CC\)
  is reflective.  It remains to observe that the projective
  objects are precisely the direct summands of countable direct
  sums of objects of~\(\CI\).  We have already observed that
  objects of~\(\CI\) are \(\VC\)\nb-projective.  Conversely,
  Theorem~\ref{the:complementary_adjoint} shows that the
  projective objects are retracts of \(\bigoplus_{n\in\N}
  \CONT_0(T\, G/H_n)\otimes A_n\) for suitable
  \(G\)\nb-\(\Cst\)-algebras~\(A_n\).  The summands are
  isomorphic to \(\Ind_{H_n}^G (\CONT_0(T_1\,G/H_n)\otimes
  A_n)\) where \(T_1\,G/H_n\) denotes the tangent space of
  \(G/H_n\) at the base point \(1\cdot H_n\).  Hence all
  projective objects are of the required form.
\end{proof}

Since the stable homological functor \(F_*(A)\defeq
\K_*(G\rcross A)\) commutes with direct sums,
Theorem~\ref{the:homological_limit} applies to it and shows
that the ABC spectral sequence for the ideal~\(\VC\) converges
towards the domain of the Baum--Connes assembly map -- which is
the localisation of~\(F_*\) at~\(\CC\) by
\cite{Meyer-Nest:BC}*{Theorem 5.2}.

It turns out that for a totally disconnected group~\(G\) the
ABC spectral sequence agrees with a known spectral sequence
that we get from the older definition of the Baum--Connes
assembly map and the skeletal filtration of a \(G\)\nb-CW-model
for the universal proper \(G\)\nb-space \(\EG G\)
(see~\cite{Kasparov-Skandalis:Buildings}).  We omit the proof
of this statement, which requires some work.

\subsection{An assembly map for torsion-free discrete quantum groups}
\label{sec:assembly-qg}

Before we turn to the assembly map, we must discuss some open
problems that lead us to restrict attention to the torsion-free
case.

The first issue is the correct definition of ``torsion'' for
locally compact quantum groups.  The torsion in a locally
compact group is the family of compact subgroups.  Quantum
groups exhibit some torsion phenomena that do not appear for
groups, and it is conceivable that we have not yet found all of
them.  First, compact quantum subgroups are not enough: they
should be replaced by proper quantum homogeneous spaces, so
that open subgroups provide torsion in \(\Cst(G)\)
whenever~\(G\) is disconnected.  Secondly, projective
representations of compact groups with a non-trivial cocycle
also provide torsion (in their discrete dual); for instance,
\(\Cst\bigl(\textup{SO}(3)\bigr)\) is not torsion-free because
of its projective representation on~\(\C^2\).

If we considered \(\Cst\bigl(\textup{SO}(3)\bigr)\) to be
torsion-free, then the Baum--Connes assembly map for it (which
we describe below) would fail to be an isomorphism.  The
correct formulation of the Baum--Connes conjecture for
\(\Cst\bigl(\textup{SO}(3)\bigr)\) turns out to be equivalent to
the Baum--Connes conjecture for
\(\Cst\bigl(\textup{SU}(2)\bigr)\) -- which is torsion-free -- so
that there is no need to discuss it in its own right
in~\cite{Meyer-Nest:BC_Coactions}.

I propose to approach torsion in \emph{discrete} quantum groups
by studying actions of its compact dual quantum group on
finite-dimensional \(\Cst\)\nb-algebras.  A discrete quantum
group is \emph{torsion-free} if any such action is a direct sum
of actions that are Morita equivalent to the trivial action
on~\(\C\).

The above definition of torsion gives the expected results in
simple cases.  First, \(\CONT_0(G)\) for a discrete group~\(G\)
is torsion-free if and only if~\(G\) contains no finite
subgroups.  Secondly, \(\Cst(G)\) for a compact group is
torsion-free if and only if~\(G\) is connected and has
torsion-free fundamental group; this is exactly the generality
in which Universal Coefficient Theorems for equivariant
Kasparov theory work (see \cites{Meyer-Nest:BC_Coactions,
  Rosenberg-Schochet:Kunneth}).  Christian Voigt shows
in~\cite{Voigt:Baum-Connes_qSU2} that the quantum deformations
of simply connected Lie groups such as \(\textup{SU}_q(n)\) are
torsion-free.

Another issue is to find analogues of the restriction and
induction functors for the non-classical torsion that may
appear, and to prove analogues of the adjointness relations
used in the proof of Theorem~\ref{the:complementary_group_BC}.
For honest quantum subgroups, the restriction functor is
evident, and Stefaan Vaes has constructed induction functors
for actions of quantum group \(\Cst\)\nb-algebras
in~\cite{Vaes:Induction_Imprimitivity}.  I expect restriction
to be left adjoint to induction for open quantum subgroups and,
in particular, for quantum subgroups of discrete quantum
groups.

For the time being, we avoid these problems and limit our
attention to the torsion-free case.  More precisely, we
consider arbitrary discrete quantum groups, but disregard
torsion.  The resulting assembly map should not be an
isomorphism for quantum groups with torsion.

The discrete quantum groups are precisely the duals of compact
quantum groups; we use reduced duals here because these appear
also in the Baum--Connes conjecture.  It is useful to
reformulate results about a discrete quantum group in terms of
its compact dual as in \cite{Meyer-Nest:Homology_in_KK}*{Remark
  2.9}.  Let~\(G\) be a compact quantum group and
let~\(\widehat{G}\) be its discrete dual.  Since we pretend
that~\(\widehat{G}\) is torsion-free, there is only one
``restriction functor'' to consider: the forgetful functor
\(\KKcat^{\widehat{G}}\to\KKcat\) that forgets the action
of~\(\widehat{G}\) altogether.  The category
\(\KKcat^{\widehat{G}}\) is equivalent to \(\KKcat^G\) by
Baaj--Skandalis duality.  Under this equivalence, the forgetful
functor \(\KKcat^{\widehat{G}}\to\KKcat\) corresponds to the
crossed product functor
\[
G\ltimes\blank\colon \KKcat^G\to\KKcat,
\qquad A\mapsto G\ltimes A.
\]
The induction functor from the trivial subgroup to~\(\widehat{G}\)
corresponds under Baaj--Skandalis duality to the functor
\(\tau\colon \KKcat\to\KKcat^G\) that equips a
\(\Cst\)\nb-algebra with the trivial action of~\(G\).  This
functor is left adjoint to the crossed product functor.

Hence the relevant subcategories \(\CI\), \(\CC\) and the
ideal~\(\VC\) correspond to
\[
\CI = \{\tau(A)\mid A\inOb\KKcat\},\qquad
\CC = \{A\inOb\KKcat^G\mid G\ltimes A\simeq0\},
\]
where~\(\simeq\) means \(\KK\)\nb-equivalence, that is,
isomorphism in \(\KKcat\), and
\[
\VC = \{f\in\KKcat^G\mid G\ltimes f=0\}.
\]

The ideal~\(\VC\) is already studied in
\cite{Meyer-Nest:Homology_in_KK}*{\S5}.  It is shown there
that~\(\VC\) has enough projective objects, and the universal
homological functor for it is described.  The target category
involves actions of the representation ring \(\Rep(G)\) of the
compact quantum group~\(G\) on objects of \(\KKcat\); such an
action on~\(A\) is, by definition, a ring homomorphism
\(\Rep(G)\to\KK_0(A,A)\).  The category \(\KKcat[\Rep(G)]\) of
\(\Rep(G)\)-modules in \(\KKcat\) is not yet Abelian because
\(\KKcat\) is not Abelian.  To remedy this, we must replace
\(\KKcat\) by its Freyd category of coherent functors
\(\KKcat\to\Ab\).  But this completion does not affect
homological algebra much because \(\KKcat[\Rep(G)]\) is an
exact subcategory that contains all projective objects; hence
we can compute derived functors without leaving the subcategory
\(\KKcat[\Rep(G)]\).

We could modify the ideal~\(\VC\) and consider all~\(f\) for
which \(G\ltimes f\) induces the zero map on \(\K\)\nb-theory.
This leads to a simpler Abelian approximation, namely, the
category of all countable \(\Z/2\)\nb-graded
\(\Rep(G)\)-modules.  But this larger ideal no longer leads to
the subcategories \(\CC\) and \(\CI\) above.

\begin{theorem}
  \label{the:complementary_BC_quantum}
  Let~\(G\) be any compact quantum group.  Then the
  ideal~\(\VC\) is compatible with countable direct sums and
  has enough projective objects.  The pair of subcategories
  \((\gen{\CI},\CC)\) is complementary.
\end{theorem}

\begin{proof}
  The ideal~\(\VC\) has enough \(\VC\)\nb-projective objects by
  \cite{Meyer-Nest:BC}*{Lemma 5.2}, which also shows that the
  \(\VC\)\nb-projective objects are precisely the direct
  summands of objects in~\(\CI\).  The ideal~\(\VC\) is
  compatible with direct sums because the crossed product
  functor commutes with direct sums.  Now
  Theorem~\ref{the:ideal_sum_complementary} shows that the pair
  of subcategories \((\gen{\CI},\CC)\) is complementary.
\end{proof}

\begin{definition}
  \label{def:assembly_map}
  Let \(F\colon \KKcat^G\to\Abel\) be some homological functor.
  The \emph{assembly map} for~\(F\) with coefficients in~\(A\) is the
  map \(\Left F(A)\to F(A)\), where the localisation \(\Left
  F\) is formed with respect to the subcategory~\(\CC\).
\end{definition}

To get an analogue of the Baum--Connes assembly map, we should
consider the functor \(F(A)\defeq \K_*(A)\) because it
corresponds to the functor \(B\mapsto \K_*(G\rcross B)\) under
Baaj--Skandalis duality.  A torsion-free discrete quantum group
has the \emph{Baum--Connes property} with coefficients if the
assembly map \(\Left F(A)\to F(A)\) is an isomorphism for
all~\(A\) for this functor.

\begin{proposition}
  \label{pro:assembly_unique}
  Let \(F\colon \KKcat^G\to\Abel\) be a homological functor
  that commutes with direct sums.  The assembly map \(\Left
  F\Rightarrow F\) is the unique natural transformation from a
  functor \(\tilde{F}\) to~\(F\) with the following properties:
  \begin{itemize}
  \item \(\tilde{F}\) is homological and commutes with direct
    sums;

  \item the natural transformation is an isomorphism  for
    objects in~\(\CI\);

  \item \(\tilde{F}\) vanishes on~\(\CC\).
  \end{itemize}
\end{proposition}

\begin{proof}
  Let \(\tilde{F}\Rightarrow F\) be a natural transformation
  with the required properties.  Since both functors involved
  are homological, the Five Lemma implies that the class of
  objects for which the natural transformation
  \(\tilde{F}\Rightarrow F\) is an isomorphism is triangulated.
  It is also closed under direct sums because both functors
  commute with direct sums.  Hence the natural transformation
  \(\tilde{F}\Rightarrow F\) is an isomorphism for all objects
  in \(\gen{\CI}\) because this holds for objects in~\(\CI\).

  Since~\(\tilde{F}\) vanishes on~\(\CC\) and is homological,
  the universal property of the localisation shows that the
  natural transformation \(\tilde{F}\Rightarrow F\) factors
  uniquely through the assembly map: \(\tilde{F}\Rightarrow
  \Left F\Rightarrow F\).  Both \(\tilde{F}\) and~\(\Left F\)
  descend to the category \(\KKcat^G/\CC\), which is equivalent
  to \(\gen{\CI}\). Since both natural transformations
  \(\tilde{F}\Rightarrow F\) and \(\Left F\Rightarrow F\) are
  invertible on objects of \(\gen{\CI}\), we get the desired
  natural isomorphism \(\tilde{F}\cong \Left F\).
\end{proof}

The critical property in Proposition~\ref{pro:assembly_unique}
is the vanishing on \(\CC\).  This cannot be expected
if~\(\widehat{G}\) has torsion.  The Baum--Connes assembly map
is an isomorphism for~\(F\) if and only if~\(F\) vanishes on
\(\CC\): one direction is trivial, and the other follows by
taking \(\tilde{F}=F\) in
Proposition~\ref{pro:assembly_unique}.  While this
reformulation of the Baum--Connes conjecture came too late to
be used in verifying the conjecture for groups, it is quite
helpful for duals of compact groups
(see~\cite{Meyer-Nest:BC_Coactions}) and probably also for
their deformations.

\section{Conclusion}
\label{sec:conclusion}

The idea of localisation -- central both in homological algebra
and in homotopy theory -- is becoming more important in
non-commutative topology as well.  When refined using
homological ideals, it unifies various new and old universal
coefficient theorems, the Baum--Connes conjecture, and its
extensions to quantum groups.

Homological ideals provide some basic topological tools in the
general setting of triangulated categories.  This includes
\begin{itemize}
\item important notions from homological algebra like
  projective resolutions and derived functors (these were
  already dealt with in~\cite{Meyer-Nest:Homology_in_KK});
\item an efficient method to check that pairs of subcategories
  in a triangulated category are complementary;
\item some control on how objects of the category are
  constructed from generators, that is, from the projective
  objects for the ideal;
\item a natural spectral sequence that computes the
  localisation of a homological functor from its values on
  generators.
\end{itemize}
Since the assumptions on the underlying category are quite
weak, all this applies to equivariant bivariant
\(\K\)\nb-theory.

We have applied this general machinery to construct the
Baum--Connes assembly map for torsion-free quantum groups,
whose domain is of topological nature in the sense that it can
be computed by topological techniques such as spectral
sequences.  But much remains to be done here.  The three main
issues are to understand torsion in locally compact quantum
groups, to adapt the reduction and induction functors to exotic
torsion phenomena, and to check whether the assembly map is an
isomorphism.  These problems are mainly \emph{analytical} in
nature.

\end{document}